\documentclass[12pt]{article}
\usepackage[margin=1.25in]{geometry}
\usepackage{pgfplots}
\usepackage{amsmath}
\usepackage{amssymb}
\usepackage{algorithmic}
\usepackage{algorithm}
\usepackage{mathtools}
\usepackage{color}
\usepackage[normalem]{ulem}

\usepackage{tikz}
\usetikzlibrary{arrows,
	arrows.meta,
	bending,calc,patterns,angles,quotes}
\usepackage{float}
\usepackage{hyperref}

\newcommand{\ep}{\varepsilon}

\newcommand{\cA}{{\cal A}}
\newcommand{\cB}{{\cal B}}

\newcommand{\cD}{{\cal D}}

\newcommand{\cI}{{\cal I}}

\newcommand{\cN}{{\cal N}}

\newcommand{\cP}{{\cal P}}

\newcommand{\cR}{{\cal R}}
\newcommand{\cS}{{\cal S}}
\newcommand{\cT}{{\cal T}}

\newcommand{\cX}{{\cal X}}

\newcommand{\cZ}{{\cal Z}}
\newcommand{\bN}{{\mathbb N}}

\newcommand{\bR}{{\mathbb R}}

\newcommand{\Ts}{\mathcal{T}}

\DeclareMathOperator*{\argmax}{argmax}

\makeatletter
\providecommand*{\cupdot}{%
	\mathbin{%
		\mathpalette\@cupdot{}%
	}%
}
\newcommand*{\@cupdot}[2]{%
	\ooalign{%
		$\m@th#1\cup$\cr
		\hidewidth$\m@th#1\cdot$\hidewidth
	}%
}
\makeatother

\newtheorem{theorem}{Theorem}

\newtheorem{lemma}{Lemma}
\newtheorem{proposition}{Proposition}
\newtheorem{remark}{Remark}
\newtheorem{definition}{Definition}

\newtheorem{corollary}{Corollary}
\def\endpf{{\ \hfill\hbox{\vrule width1.0ex height1.0ex}\parfillskip 0pt
	}}
\newenvironment{proof}{\noindent{\bf Proof:}}{\endpf}
\newcounter{figurecounter}
\begin{document}
\setlength{\baselineskip}{20pt}

\title{Bayesian games with nested information%
\thanks{
We thank three anonymous referees for their valuable input, which helped us significantly improve the paper. We also thank Ziv Hellman, Ehud Lehrer, Dov Samet, and Omri Solan for useful discussions. Jacobovic acknowledges the support of the Israel Science Foundation, grant \#3739/24. Solan acknowledges the support of the Israel Science Foundation, grant \#211/22.}}

\author{Royi Jacobovic\footnote{School of Mathematical Sciences, Tel-Aviv University, Tel-Aviv, Israel, 6997800, E-mail: royijacobo@tauex.tau.ac.il.},\and
John Yehuda Levy\footnote{Adam Smith Business School, University of Glasgow, 2 Discovery Place, Glasgow, G11 6EY, UK, E-mail: john.levy@glasgow.ac.uk},
\and Eilon Solan\footnote{School of Mathematical Sciences, Tel-Aviv University, Tel-Aviv, Israel, 6997800, E-mail: eilons@tauex.tau.ac.il.}}

\maketitle

\begin{abstract}
\setlength{\baselineskip}{20pt}
A Bayesian game is said to have nested information if the players are ordered,
and each player knows the types of all players that follow her in that order.
We prove that all multiplayer 
Bayesian games with finite actions spaces, 
bounded payoffs,
Polish type spaces, 
and nested information admit a Bayesian equilibrium. 


\end{abstract}



\section{Introduction}

Although models of incomplete information are 
abundant
in economic modeling, general results on existence of equilibrium are hard to come by. 
In particular, 
while often it is natural to 
model the possible rewards and beliefs of agents  using a continuum of possibilities, 
there are relatively few  general 
existence
results at our disposal, and they use fairly restrictive assumptions on the players information. 
In this paper, we examine games under a naturally arising information structure,
namely, \emph{nested information}, where the
players can be ordered 
according to the amount of information they possess, 
from the most knowledgeable to the least knowledgeable.

Harsanyi \cite{Harsanyi1967} laid the foundation of games of incomplete information, also known as Bayesian games, which have greatly influenced the development of game theory. 
In that model, each agent has a \emph{type}, which includes 
her
belief about payoffs, others' beliefs about the payoffs, others' beliefs about others' beliefs about the payoffs, and so forth. 
There is a prior over the possible type profiles that may occur. 
Each agent is informed of 
her
own type, and must choose 
her
policy as a function of it. Payoffs are a function of types and actions, and agents try to maximize their expected utilities, given the strategies of the others, leading to the notion of \emph{Bayesian equilibrium}, the natural generalization of \emph{Nash equilibrium} to the incomplete-information setup.

While games in which agents may have only finitely many types pose no difficulty for equilibrium existence, 
when there is a continuum of types the situation becomes much thornier. 
These are standards frameworks in economic modeling, as it is natural and convenient to allow, e.g., prices, quantities, and profits, to assume any value (within some range). 
However, the use of 
a continuum of types
makes it extremely difficult to show that Bayesian equilibrium must exist. 

One of the very few general existence results is Milgrom and Weber \cite{MW1985},
who 
assumed that the prior belief of the agents is either independent across types, or at least absolutely continuous with respect to some independent prior (i.e., absolutely continuous with respect to the product of the marginals). 
It remained for some time an open question as to whether, failing this condition, equilibria could fail to exist.
Simon \cite{Simon2003} 
showed
that this was indeed the case
by constructing
an example of a Bayesian game with a continuum of types and no Bayesian equilibrium. 
Hellman \cite{Hellman2014} provided an example of a two-player Bayesian game with finite action spaces and no Bayesian  $\ep$-equilibrium for all $\ep > 0$ sufficiently small;
that is, for every strategy profile, 
 a positive probability of types of one of the players can profit more than $\ep$ by deviating. 
Simon and Tomkowicz \cite{Simon2018, Simon2023} provided examples of, respectively, three-player and two-player Bayesian games that do not admit a Harsanyi $\ep$-equilibrium for $\ep > 0$ sufficiently small; that is, for every strategy profile, at least one player can profit more than $\ep$ by deviating at the ex-ante stage game.

The information structure we examine in this paper is that of \emph{nested information}; that is, 
the players can be ordered from most knowledgeable to least knowledgeable. The most knowledgeable player (say, Player 1) knows everything Player 2 knows (and possibly more), Player 2 knows everything Player 3 knows (and possibly more), etc. 
Such structures have been modeled in hierarchical organization paradigms, financial market games, persuasion models, and others; we recall some of these works and more below. 
Such games generally do not satisfy the absolute continuity condition of Milgrom and Weber \cite{MW1985}.
For instance, if there are three agents, two of which are informed of a value $v$ in some range $[\underline{v},\overline{v}]$ which distributes continuously while the third is not, the possible type profiles distribute continuously along a diagonal,
and do not satisfy the absolute continuity condition.

In this work, 
we study these games when players have finitely many actions at their disposal. 
Our main result shows that in such games, Bayesian equilibria do exist, 
thereby exhibiting an additional class of incomplete information games possessing equilibria. 
(A discussion on models with a continuum of actions appear in Section~\ref{section:discussion}.) 
Our proof introduces two new tools to the study of Bayesian games, which may
prove useful also for other classes of games as well as other questions on Bayesian
games. 

The first tool, used for establishing existence of Bayesian $\varepsilon$-equilibrium, is a finite approximation of the belief hierarchy. 
As is well known, the players' beliefs about the payoffs, 
the players' belief about the others' beliefs, 
the players' belief about the others' beliefs about the others' beliefs, and so forth, form an infinite hierarchy. 
When information is nested, as we will elaborate
below,
this infinite hierarchy is determined by the first $n$-orders of the ladder, where $n$ is the number of players. 
This finiteness of the relevant levels will allow us to 
construct a finite approximation of the space of infinite hierarchies that is sufficient for our purpose: We will define an approximating Bayesian game whose finite type spaces are induced by this finite approximation,
and show that 
\color{black}
a Bayesian equilibrium of this game,
which exists by 
\cite{Harsanyi1967}, yields a Bayesian $\varepsilon$-equilibrium of the original game.

The second tool, used for establishing existence of Bayesian equilibrium, is 
the
Measurable ``Measurable Choice” Theorem 
by
Mertens \cite{Me87}, 
a tool which had previously been used in the study of stochastic games but is novel in its application to Bayesian games.\footnote{An application of the Measurable ``Measurable Choice'' Theorem to a one sender - many receivers game can be found in Zeng \cite{Zeng 2023}.} 
To construct a Bayesian equilibrium
we would like to take a limit of Bayesian $\varepsilon$-equilibria as $\varepsilon$
goes to
$0$. 
However, 
it is well known that in the limit, correlation 
may be introduced; see Stinchcombe \cite{Stinchcombe2011}. Conceptually, we construct the equilibrium among the limit points of a sequence of Bayesian $\frac{1}{n}$-equilibria, step-by-step, starting from the least knowledgeable player, 
and for each player we need to use a 
purification result to guarantee appropriate consistency with the selections already chosen. Not only is the purification done repeatedly, but it needs to be done a continuum-many times at each stage, all in a measurable fashion; 
this is precisely where the Measurable “Measurable Choice” 
Theorem comes into play.

\paragraph*{Structure of the paper.}

The paper is organized as follows. Section~\ref{sec:lit} discusses related literature on Bayesian games and on nested information structures.
Section~\ref{section:model} is devoted to the presentation of the model and the main result. Section~\ref{sec:driving} gives heuristic overviews of the proofs. The proof of the main result appears in Section~\ref{section:proof}. Further discussion and open problems appear in Section~\ref{section:discussion}.


\section{Literature on Bayesian Equilibrium and on Nested Information}\label{sec:lit}

\paragraph*{Bayesian Equilibria} 
Since Bayesian (and even Harsanyi) $\ep$-equilibria need not exist in Bayesian games,
it is important to find sufficient conditions on the parameters of the game that ensure 
they exist. A large literature expanded the sufficient conditions identified by Harsanyi \cite{Harsanyi1967} and Milgrom and Weber \cite{MW1985}.

Stinchcombe and White \cite{Stinchcombe1992}  proved that 
when all players share the same information,
or when there are two players and information is nested, a Harsanyi equilibrium exists.
Ui \cite{Ui2016} proved the existence and uniqueness of Bayesian equilibrium in Bayesian games
where the payoff function is continuously differentiable on the action space for each vector of types,
and its gradient satisfies certain conditions.
Hellman and Levy \cite{Hellman2017, Hellman2019} studied 
Bayesian games with purely atomic types;
they showed that a Bayesian equilibrium exists when the common knowledge relation is smooth,
namely, the common knowledge classes are level sets of a Borel function.
Moreover,
for any common knowledge relation that is not smooth, there exists a type space that yields this common knowledge relation and payoff functions such that the resulting  Bayesian game does not have a Bayesian $\ep$-equilibrium, provided $\ep$ is sufficiently small.
Carbonell-Nicolau and McLean \cite{Carbonell2018} extended the result of Milgrom and Weber \cite{MW1985} to Bayesian games with general action sets, by requiring that the payoff functions are upper semi-continuous and satisfy a condition related to Reny's uniform payoff security (Reny \cite{Reny 1999}).
Olszewski and Siegel \cite{Olszewski2023}
simplified the application of Reny’s~\cite{Reny 1999} better-reply security to
Bayesian games where players' types are independent, and used this condition to prove the existence of Harsanyi equilibria for
classes of games in which payoff discontinuities arise only at ``ties.'' 

Several papers provided sufficient conditions that guarantee the existence of a \emph{pure} Bayesian equilibrium, see, e.g., 
Radner and Rosenthal \cite{Radner 1982},
Vives \cite{Vives 1990}, 
Khan and Sun \cite{Khan 1999}, Reny \cite{Reny 2011},
and Barelli and Duggan \cite{Barelli 2015}. While \cite{Radner 1982},
assumed the players' types are independent,
and \cite{Barelli 2015} made the more general assumption of \cite{MW1985} regarding the absolute continuity of the joint distribution of types,
the other works do not make these assumptions. Existence of equilibria in Bayesian games with
\emph{infinitely} many players was studied by, e.g.,
Kim and Yannelis \cite{Kim1997}, 
Balbus et al.~\cite{Balbus 2015}, and Yang \cite{Yang 2022}. 

It is interesting to note that 
Bayesian games with nested information can be recast as regular projective games (see Myerson and Reny \cite{Myerson2020}, Section 9).
It follows from Theorem 9.3 in Myerson and Reny \cite{Myerson2020} that 
for every $\ep > 0$, 
Bayesian games with nested information admit
a Bayesian $\ep$-equilibrium 
under proper technical conditions,
which include the continuity of the payoff function over the type space
and the fact that the type distribution has a continuous density function.
Our paper strengthens \cite{Myerson2020} by
(i) proving the result for $\ep = 0$,
and 
(ii) weakening the conditions required to derive the existence result
(while requiring that the set of actions is finite rather than general) 
by dropping the requirement that the prior has a density with respect to a product distribution on types, which, would not be satisfied in many nested information structures of interest.

\paragraph*{Nested Information}

Nested information arises naturally in
strictly hierarchical organizations, 
where higher-level managers have
more information than lower-level managers and workers. 
For example, Mathevet and Taneva \cite{MT 2022} considered a game
where these information structures arise endogenously. 
In their case,
before the agents simultaneously take their payoff-relevant action, there is cheap-talk communication between the players following the strict
order prescribed by the hierarchy; that is, each agent sends messages to the agent immediately below her. 
It can be shown that 
regardless of the exact information transmitted, the information structure endogenously generated in the cheap-talk stage will be nested.

Nested information also arises naturally in situations where players obtain or are exposed to different levels of information. 
For example, managers of firms are more informed about the firm's financial situation than large investors, who in turn are more informed than small investors.
Experiments on financial market games
where investors can predict future dividends or future value of a certain stock for different spans of time have been reported by, e.g.,
Toth \cite{Toth} and Huber \cite{Huber}. 

Another model with nested information
is when 
players are divided into two subsets:
those who obtain symmetric information about the state of the world,
and those who are completely ignorant about it.
For example, 
Debo et al.~\cite{Debo2012} and Kremer and Debo \cite{Kremer2016}
study service system that can provide service in various qualities. Customers have two possible types:
some know the service quality,
while the others are not exposed to this information. 
Additional literature on service systems with similar features can be found in Hassin~\cite{Hassin2016}. 

A more general information structure is considered in Wu et al.~\cite{Wu2021},
who study a routing model where the players are divided into groups, 
and the players in each group obtain the same information on the state of nature.
When the number of groups is $2$ and the players in one of the groups obtain no information,
or when the signals that the groups share are nested,
this model exhibits nested information as well.

\color{black}

Finally, nested information arises in two-player models, where one player is more informed than the other,
like dynamic games with asymmetric information (e.g., Aumann and Maschler \cite{aum, aum95}, Cardaliaguet and Rainer \cite{Cardaliaguet}, Gr\"un \cite{Grun}, De Angelis et al.~\cite{De Angelis2022b}, and Jacobovic \cite{Jacobovic2022}),
and Bayesian persuasion models (e.g., Kamenica and Gentzkow \cite{Kamenica2011} and Kamenica \cite{Kamenica2019}).
\color{black}

\section{The Model and Main result}
\label{section:model}

\noindent\textbf{Notations.}
Let $\cN = \{1,2,\ldots,n\}$, with $n$ finite. Whenever $(X_i)_{i \in \cN}$ is a collection of sets,
we denote their Cartesian product by $X \equiv \prod_{i \in \cN} X_i$;
for $j \in \cN$ we denote $X_{-j} \equiv \prod_{i \in \cN \setminus \{j\}} X_i$.
For any $1\leq j_1\leq j_2\leq n$ denote
$[j_1:j_2] \equiv \{j_1,j_1+1,j_1+2,\ldots,j_2\}$
and $X_{j_1:j_2}\equiv\prod_{i=j_1}^{j_2}X_i$.
A product of measurable spaces will always be considered a measurable space with the product $\sigma$-field.
Whenever $x = (x_i)_{i \in \cN}$ is a vector and $j \in \cN$, 
we set $x_{-j} \equiv (x_i)_{i \in \cN \setminus \{j\}}$, and 
for any $1\leq j_1\leq j_2\leq n$ we set $x_{j_1:j_2}\equiv\left(x_{j_1},x_{j_1+1},\ldots,x_{j_2}\right)$.
When $(U_i)_{i=1}^n$ are real-valued functions,
we denote by $U_{j_1:j_2}$ the vector-valued function $(U_{j_1},U_{j_1+1},\dots,U_{j_2})$.

For every measurable set $X$, we denote by $\Delta(X)$ the set of probability distributions on $X$.
We consider $\Delta(X)$ as a topological space, e.g., by endowing it with some metric like the total variation metric or the Prokhorov metric. 
When $X$ and $Y$ are two random variables, 
we say that \emph{$Y$ is determined by $X$} if there exists a measurable function $\kappa(\cdot)$ such that $\kappa(X)=Y$ with probability one.
\bigskip

\begin{definition}[Bayesian game]
\label{def:bayesian:game}
A \emph{Bayesian game} $\Gamma$ is given by
\begin{itemize}
\item A finite set of players $\mathcal{N}\equiv\left\{1,2,\ldots,n\right\}$, for some  $n \geq2$.

\item For each $i\in\mathcal{N}$, a topological type space $\mathcal{T}_i$.

\item A common prior distribution ${\mathbb P}$ on $\mathcal{T}\equiv \prod_{i \in \cN} \cT_i$.


\item For each $i\in\mathcal{N}$, 
a topological space $\cA_i$ of actions.

\item 
For each $i \in \cN$, a measurable payoff function $R_i : \mathcal{T} \times \cA \to \bR$.
For each $i \in \cN$,
we denote by $R_i(a) : \mathcal{T} \to \bR$ the $a$-section of $R_i$, for each $a \in \cA$;
and by $R_i(t) : \cA \to \bR$ the $t$-section of $R_i$, for each $t \in \cT$.
We also set $R \equiv (R_i)_{i \in \cN}$.
\end{itemize}
\end{definition}

We will denote by $t = (t_1,t_2,\dots,t_n)$ a random type vector,
so that $t_i$ is the random type of Player~$i$.

For every $i\in\mathcal{N}$, denote by $\mathcal{X}_i \equiv \Delta(\cA_i)$ the set of mixed actions of Player~$i$.
A (behavior) \emph{strategy} of Player~$i$ is a measurable function $s_i : \cT_i \to \cX_i$. 
This definition indicates the interpretation of the type spaces:
each player $i \in \cN$ knows her own type,
and is not told the types of the other players.
Denote by $\cS_i$ the set of strategies of Player~$i$,
so that $\cS \equiv \prod_{i \in \cN} \cS_i$ is the set of all strategy profiles.

Every strategy profile $s \in \cS$ induces a probability distribution  over $\cT \times \cA$, denoted ${\mathbb P}_s$, which satisfies
\begin{equation}\label{eq: product form}
{\mathbb P}_s(T \times B_1 \times \dots \times B_n) 
= \int_{T} \prod_{i \in \cN} s_i(t_i)(B_i) \mathrm{d} {\mathbb P}(t),
\end{equation}
for every Borel set $T \subseteq \cT$ and every measurable sets $B_i \subseteq \cA_i$ for $i \in \cN$.
Denote by $E_s$ the corresponding expectation operator.
For every $s\in\mathcal{S}$, whenever the expectation of $R_i$ with respect to ${\mathbb P}_s$ is well defined,
Player~$i$'s expected payoff under the strategy profile $s$ is the real number 
\[ U_i(s) \equiv {\mathbb E}_s[R_i],\]
and her conditional payoff given her information is the random variable (determined by $t_i$)%
\footnote{Here for simplicity we abuse notation.
Formally, $U_i(s \mid t_i)$ is the conditional expectation 
of $R_i$ given the sigma-field $\cB(\cT_i) \times \prod_{j \neq i} \{\cT_{j},\emptyset\}$.}
\[ U_i(s\mid t_i) \equiv {\mathbb E}_s[R_i \mid t_i].\]
The solution concept we will concentrate on in this paper is Bayesian $\ep$-equilibrium.

\begin{definition}[Bayesian $\ep$-equilibrium]
    Given $\ep \geq0$,
a strategy profile $s^*\in\mathcal{S}$ is
   a \emph{Bayesian $\ep $-equilibrium} if for every player $i \in \cN$ and every strategy $s_i \in \mathcal{S}_i$,
    \begin{equation}
    \label{equ:bayesian:eq}
        U_i(s_i,s^*_{-i} \mid t_i)\leq U_i(s^* \mid t_i)+\ep, \ \ {\mathbb P}\text{-a.s.}
    \end{equation}
\end{definition}

\begin{remark}[The relation between Bayesian and Harsanyi equilibria]\normalfont
\label{remark:bayesian:harsanyi} 
Given $\ep \geq0$,
a strategy profile $s^*\in\mathcal{S}$ is
   a \emph{Harsanyi  $\ep $-equilibrium} if 
    for every player $i \in \cN$ and every strategy $s_i \in \mathcal{S}_i$,
    \begin{equation}
        U_i\left(s_i,s^*_{-i}\right)\leq U_i\left(s^*\right)+\ep.
        \end{equation}
Standard conditioning implies that every Bayesian $\ep$-equilibrium is a Harsanyi $\ep$-equilibrium. 
 When  $\left(\cT,\mathcal{B}\left(\mathcal{T}\right),{\mathbb P}\right)$ is complete, 
a Harsanyi $0$-equilibrium is also a Bayesian $0$-equilibrium. 
When $\varepsilon>0$, a Harsanyi $\ep$-equilibrium is not necessarily a Bayesian $\ep$-equilibrium.
Indeed, modifying a Bayesian $0$-equilibrium on a set of types of sufficiently small measure arbitrarily will generically yield such an example.
In fact, the example provided by Hellman~\cite{Hellman2014}
shows that when $\ep>0$ is sufficiently small, Harsanyi $\ep$-equilibria may exist while Bayesian $\ep$-equilibria do not.
See Hellman and Levy~\cite{Hellman2021} for further discussion on this issue.
\end{remark}

Harsanyi \cite{Harsanyi1967} first presented the model of Bayesian games,
and proved that when all sets that define the game are finite, a Bayesian \color{black} equilibrium exists.
Milgrom and Weber \cite{MW1985} studied Bayesian games with general type spaces,
and 
proved that a Harsanyi equilibrium exists in distributional strategies when ${\mathbb P}$ is absolutely continuous w.r.t.~the product of its marginals,
that is, 
w.r.t.~${\mathbb P}_1 \otimes {\mathbb P}_2 \otimes \dots \otimes {\mathbb P}_n$,
where ${\mathbb P}_i$ is defined by 
${\mathbb P}_i(B_i)\equiv {\mathbb P}\left(\left(\prod_{j\neq i}\mathcal{T}_j\right) \times B_i\right)$ for every $i\in\mathcal{N}$ and every $B_i \in \cB(\cT_i)$. As mentioned in the introduction,
Simon \cite{Simon2003} and Hellman \cite{Hellman2014} 
(resp., Simon and Tomkowicz \cite{Simon2018, Simon2023}) provided examples of  Bayesian games with finite action spaces and no Bayesian $\ep$-equilibria (resp., no Harsanyi $\ep$-equilibria), for $\ep > 0$ sufficiently small. Additional sufficient conditions on the parameters of the game that ensure the existence of a $0$-equilibrium have already been reviewed in the introduction.

In this paper we concentrate on Bayesian games where the information of the players is nested.

\begin{definition}[Nested information]
\label{def:nested}
We say that the information of the players in a Bayesian game is \emph{nested} if $t_{i+1}$ is determined by $t_i$,
for every $i=1,2,\ldots,n-1$; that is, 
if for each $i < n$ there is a mapping $\kappa_i : \cT_i \to \cT_{i+1}$ such that 
\begin{equation}
{\mathbb P}\bigl(t_{i+1} = \kappa_i(t_i), \ \ \ \hbox{for every } 1\leq i<n\bigr)  =1.    
\label{equ:nested}
\end{equation}
\end{definition}

\begin{remark}\normalfont[Nested information w.r.t.~a general ordering of the players]
More generally, we could say that the information of the players is nested if the 
condition in Definition~\ref{def:nested}
holds after a permutation of the players. For simplicity we will always assume the ordering of the players in Definition~\ref{def:nested}, with Player 1 the most knowledgeable, followed by Player 2, and so forth.
\end{remark}

\begin{remark}[Players possessing the same information]\normalfont
Note that the definition allows for two or more players to possess the same information.
Indeed, players $i$ and $i+1$ have the same information if the function $\kappa_i$ in Eq.~\eqref{equ:nested} is a bijection.
\end{remark}

\begin{remark}[${\mathbb P}$-a.s.~versus everywhere in Eq.~\eqref{equ:nested}]
\label{rem:nested_info_everywhere}\normalfont
Nested information requires that 
$t_{i+1} = \kappa_i(t_i)$, ${\mathbb P}$-a.s.~and not everywhere.
This distinction is irrelevant for our purposes.
Indeed,  let ${\mathbb Q}$ be the measure on $\prod_{i \in I} \cT_i$ whose marginal on $\cT_1$ coincides with that under ${\mathbb P}$,  and that is determined by its marginal on $\cT_1$ and the functions $\kappa_1,\ldots,\kappa_{i-1}$. 
We then have ${\mathbb P}={\mathbb Q}$.
\end{remark}

\begin{remark}[Nested information and absolute continuity of information]
\normalfont
We here show that a Bayesian game  with nested information may not satisfy the requirement of absolute continuity of information structure, as studied by 
Milgrom and Weber \cite{MW1985}.
Indeed, suppose that $\cT_i = [0,1]$ for each $i \in \cN$,
and ${\mathbb P}$ is the uniform distribution on the diagonal $\{t_1 = t_2 = \dots = t_n\}$.
The resulting measure $\bigotimes_{i=1}^n {\mathbb P}_i$ is the Lebesgue measure on $[0,1]^n$,
and  hence  ${\mathbb P}$ is concentrated on a set of $(\bigotimes_{i=1}^n {\mathbb P}_i)$-measure zero. 
Therefore, this information structure does not satisfy the absolute continuity condition of \cite{MW1985}, yet the players have nested information.

Since Bayesian games that satisfy absolute continuity of information structures do not necessarily have nested information 
  (see, e.g., Example~2 in \cite{MW1985}),
it follows that 
nested information 
 is unrelated to absolute continuity of information structure.
\end{remark}

Below we will study Bayesian games that satisfy (a subset of) the following regularity conditions:
\begin{description}

\item [\textbf{A1}:] (Discrete action spaces and bounded payoffs\color{black}) $\mathcal{A}_i$ is finite and $R_i$ is bounded, for every $i\in\mathcal{N}$; that is,
there is $M > 0$ such that $|R_i(t,a)| \leq M$ for every $i \in \cI$, $t \in \cT$, and $a \in \cA$. 

     \item[A2:] (Polish type spaces) For each $i\in\mathcal{N}$, the type space $\mathcal{T}_i$ is Polish.%
\footnote{Recall that a \emph{Polish space} is a separable completely metrizable topological space; that is, a space homeomorphic to a complete metric space that has a countable dense subset.}
\end{description}


The main result of the present work is the following.
\color{black}

\begin{theorem}[Existence of 0-equilibrium]
\label{thm:0}$\text{ }$
Every Bayesian game with nested information that satisfies Conditions \textbf{A1} and \textbf{A2} admits a Bayesian $0$-equilibrium.
\end{theorem}
\color{black}



\begin{remark}[Tightness of the conditions in Theorem~1]
\normalfont 
The finiteness of the sets of actions,
or at least some compactness condition, 
is required to ensure the existence of an equilibrium even when the game has complete information.
The boundedness of the payoffs is required to ensure that the expected payoff of the players is well defined.
The example of Hellman \cite{Hellman2014},
which does not admit a Bayesian $\ep$-equilibrium for all $\ep > 0$ sufficiently small,
satisfies \textbf{A1} and \textbf{A2} and does not have nested information. 
In particular, \textbf{A1} and the assumption that information is nested cannot be dispensed with in Theorem~\ref{thm:0}.
We do not know whether \textbf{A2} is essential to prove Theorem~\ref{thm:0}.
In Section~\ref{section:discussion} we discuss possible weakening of \textbf{A1}.
\end{remark}
\color{black}
\begin{remark}[Games with two players and nested information]
\normalfont Example 3.3 in Stinchcombe and White \cite{Stinchcombe1992} implies that in the presence of two players who have nested information (and concave payoffs in their actions),  under Conditions \textbf{A1}  and \textbf{A2}, 
there exists a Harsanyi $0$-equilibrium.  As mentioned in Remark~\ref{remark:bayesian:harsanyi},
when $(\cT,\cB(\cT), {\mathbb P})$ is complete,
this implies the existence of a Bayesian 0-equilibrium.
Thus, 
Theorem \ref{thm:0} extends the result of Stinchcombe and White \cite{Stinchcombe1992}
to any number of players.
\color{black}
\end{remark}
As mentioned in the introduction,
when the game does not have nested information,
Theorems~\ref{thm:0} may fail.






\section{The Driving Force Behind our Proofs}\label{sec:driving}

This section gives heuristic explanations of our methodology. The proof is divided into two main parts: 
The first establishes the existence of $\varepsilon$-equilibria, while the second uses the Measurable ``Measurable Choice'' Theorem to construct an appropriate limit of approximate equilibria which constitute an exact equilibrium. As remarked earlier, we point out that it is a well-known problem that in games with a continuum of states, limits of approximate equilibria do not, in general, naturally induce exact equilibria, as the limiting process can induce correlation; for an elaboration on this point, see Stinchcombe \cite{Stinchcombe2011}.

\paragraph*{Belief Hierarchies}
In Bayesian games, 
to determine her action,
on top of her own information on the players' types,
a player needs to take into account also:
\begin{enumerate}
    \item[-] her information on the information the other players have on the players' types,
    \item[-] her information on the information each Player $i$ has on the information each Player $j  \neq i$ has on the players' types,
    \item[-] her information on the information each Player $i$ has on the information each Player $j  \neq i$ has on the information each Player $k \neq j$ has on the players' types,
    \item[-] etc. 
\end{enumerate}
In general,
the information encapsulated in higher levels cannot be deduced 
from the information encapsulated in lower levels. 
This gives rise to an \emph{infinite belief hierarchy},
which typically depends on the state of the world.

The belief hierarchy of a player identifies the states of the world which should be taken into account when determining the player's action.
In fact, the belief hierarchies of the players divide the set of states of the world into disjoint subsets,
called \emph{minimal belief subspaces},
such that the states in each subspace are closed, in the sense that when the actual state of the world is in a given subspace,
only states in that subspace need to be considered to determine the players' actions in equilibrium.
As showed by Simon \cite{Simon2003},
even if the game restricted to each of the minimal belief subspaces has an equilibrium,
the amalgamation of these equilibria need not be measurable.
The example provided by Hellman \cite{Hellman2014}
does not possess an $\ep$-equilibrium when $\ep > 0$ is sufficiently small.

When information is nested, the infinite belief hierarchies can be deduced from its first $n$ levels,
where $n$ is the number of players. Indeed, if, say, there are two players and Player~1 is more informed than Player~2,
then Player~2's infinite belief hierarchy can be deduced from her own information on the players' types and her information on Player~1's information on the players' types;
and Player~1's infinite belief hierarchy can be deduced from her own information on the players' types, and Player~2's information on the players' types and on Player~1's information on the players' types. For example, the next level in the belief hierarchy of Player~2 corresponds to Player~2's information on Player~1's information on Player~2's information on the players' types,
which coincide with Player~2's information on the players' types.
Similarly, 
the next level in the belief hierarchy of Player~1 corresponds to Player~1's information on Player~2's information on Player~1's information on the players' types,
which coincide with Player~2's information on Player~1's information on the players' types.
\color{black}
Thus, when information is nested, there is no need to consider infinite belief hierarchies,
and the game has a finite structure.

\paragraph*{$\varepsilon$-Equilibrium: Approximating Belief Hierchies}

The observation made in the previous paragraph leads us to define a finite approximation of the belief hierarchy when information is nested, which is useful in proving the existence of a Bayesian $\ep$-equilibrium. Let us explain this approximation.

Assume that the action spaces are finite and suppose again that there are two players, where Player~1 is more informed than Player~2.
When Player~1 observes the type profile realization $t=(t_1,t_2)$, Player~1 has a belief over the matrix game that is being played.
As at present we are interested in an $\ep$-equilibrium, we can assume that the collection $\cR$ of all possible matrix games is finite.
Fixing $\delta > 0$, we can choose a $\delta$-dense subset $\cD_1$ of 
the set $\Delta(\cR)$ of probability distributions over $\cR$,
and approximate Player~$1$'s belief at $t$ by the closest point in $\cD_1$, denoted $\varphi_1(t)$.
We can then consider the mapping $\psi_1$ that assigns to each $t$ the pair $(\varphi_1(t),R(t))$, namely, Player~1's approximated belief at $t$ and the payoff matrix at $t$,
and consider the distribution of this vector given $t_2$ which is Player~2's information at $t$.
Since the mapping $\varphi_1$ takes only finitely many values, 
and the number of possible payoff matrices is finite,
the range of $\psi_1$ is finite dimensional, and hence can be in turn $\delta$-approximated by a mapping $\varphi_2$ with finitely many values; 
the range of $\varphi_2$ is a $\delta$-dense subset $\cD_2$ of $\Delta(\cD_1 \times \cR)$.
The mapping $\varphi_2$ represents the approximated information Player~2 has at $t$,
on both the payoff matrix and on Player~1's information on the payoff matrix.
Finally, we say that Player~1's approximated belief is composed by the pair $(\varphi_1,\varphi_2)$,
and Player~2's approximated belief is composed solely of $\varphi_2$.

Since the approximating information divides the state space into finitely many sets, 
the resulting game admits a Bayesian 0-equilibrium.
The properties of the approximation then imply that this Bayesian 0-equilibrium is a Bayesian $\ep$-equilibrium, provided $\delta$ is sufficiently small. 

 \paragraph*{Exact Equilibrium: Using the Measurable ``Measurable Choice'' Theorem.}

To construct a Bayesian 0-equilibrium,
we would like to consider an accumulation point of a sequence of Bayesian $\frac{1}{n}$-equilibria as $n \rightarrow \infty$.
Unfortunately, as mentioned 
above, 
when the type space is general, a limit of strategy profiles may be a correlated strategy profile. 
To 
overcome this difficulty,
we use iteratively an extension of Mertens' \cite{Me87} Measurable ``Measurable Choice'' Theorem.
This is the first time we are aware of where this theorem is used in the study of Bayesian games.%

To illustrate the construction,
let us start with
three players. There is a single coordinate of knowledge, which Players 1 and 2 know,
and
Player 3 does not know; i.e., $\Ts_1 = \Ts_2$, while $\Ts_3$ is a singleton $\emptyset$; and the prior ${\mathbb P}$ is concentrated on the set $\{ t_1 = t_2\}$, i.e., $\kappa_1(t_1) = t_1$, and $\kappa_2(t_2)= \emptyset$. 
For simplicity, assume 
each player has two actions, $L$ and $R$. 
In this
framework, equilibrium existence had been previously an open question.

Fix a sequence of approximate equilibria $(s^n)_{n \in \bN}$, where $s^n$ is a Bayesian $\frac{1}{n}$-equilibrium for every $n\in \bN$. 
A natural approach is the following: 
First, consider a sequence of indices $(n_k)_{k \in \bN}$ such that $s_3^{n_k}(\emptyset)$ converges to some limit, 
denoted  $s_3^*(\emptyset)$,
which we  
designate as a candidate for 
the equilibrium strategy for Player 3. 
Next, for each possible $t_1 \in \Ts_1 = \Ts_2$ (we treat $s_2$ as a function of $t_1$ as well, as a.s. $t_1=t_2$), let $s_{1,2}^*(t_1)= (s^*_1(t_1),s^*_2(t_1))$, the 
candidate for the
equilibrium strategy of Players 1 and 2, be 
an accumulation
point of $\big( s_{1,2}^{n_k}(t_1) \big)_{k \in \bN}$,
where $s_{1,2}^{n_k}(t_1) = (s_{1}^{n_k}(t_1),s_{2}^{n_k}(t_1))$;
\color{black}
we would make the selections so that $s_{1,2}^*$ is a measurable function.

The problem with allowing $s_{1,2}^*(t_1)$ to be any 
accumulation
point of $(s_{1,2}^{n_k}(t_1))_{k \in \bN}$ is that this may result in the expected payoffs to Player~1's actions not being able to support $s^*_3(\emptyset)$ as an equilibrium. 
For instance, $s^*_3(\emptyset)$ may mix between $L$ and $R$, but the expected payoffs for $L$ and for $R$ assuming Players $1$ and $2$ use $s_{1,2}^*$, may be different, i.e., we may have $U_3(L, s^*_{1,2} \mid \emptyset ) \neq U_3(R, s^*_{1,2} \mid \emptyset )$. 
To avoid this problem, we need $(n_k)_{k \in \bN}$ to be such that not only does $(s_3^{n_k}(\emptyset))_{k \in \bN}$ converge, but so do the sequences $\big(U_3(L, s^{n_k}_{1,2} \mid \emptyset )\big)_{k \in \bN}$ and $\big( U_3(R, s^{n_k}_{1,2} \mid \emptyset ) \big)_{k \in \bN}$.
Denote the corresponding limits by
$\rho_3[L]$ and $\rho_3[R]$. 
The selection of $s_{1,2}^*(t_1)$ over $t_1 \in \Ts_1$ is done among the 
accumulation
points of $(s_{1,2}^{n_k}(t_1))_{k \in \bN}$ so that, in the aggregate, $U_3(L, s^*_{1,2} \mid \emptyset ) = U_3(R, s^*_{1,2} \mid \emptyset )$. 
The resulting selections together with $s^*_3$ can be shown to be an equilibrium.

To be more precise, and of particular relevance as the examples become more complex and we look towards a general technique, is that we do \emph{not} actually fix a specific subsequence of indices $(n_k)_{k \in \bN}$; what we do is stipulate that in the latter stage, when we select among 
accumulation
points of $(s_{1,2}^n(t_1))_{n \in \bN}$, we only select among limits of subsequences whose indices $(n_k)_{k \in \bN}$ ensure not only the convergence $s_3^{n_k}(\emptyset) \rightarrow s_3^*(\emptyset)$, but \emph{also} the convergence $U_3(L, s^{n_k}_{1,2} \mid \emptyset ) \rightarrow \rho_3[L]$ and $U_3(R, s^{n_k}_{1,2} \mid \emptyset ) \rightarrow \rho_3[R]$. There may be many such subsequences $(n_k)_{k \in \bN}$, and we do not select a particular one. 
Rather, we select the limits we desire, and make sure that future selections are consistent with these limits. 

Now, let us spice up the example. Suppose there are 
four players:
Player 4 knows nothing, Player 3 knows something,
and Players 1 and 2 know everything. 
Formally, $\Ts_1 = \Ts_2$ and a.s.~$t_1=t_2$ (that is, $\kappa_1(t_1)=t_1$), $\Ts_3$ is non-trivial (there is some $\kappa_2:\Ts_2 \rightarrow \Ts_3$), while $\Ts_4$ is a singleton $\emptyset$ ($\kappa_3(t_3) = \emptyset$). 
Again, assume players have two actions $L$ and $R$. We fix a sequence of approximate equilibria $(s^n)_{n \in \bN}$ with $s^n$ being a Bayesian $\frac{1}{n}$-equilibrium. 
Once again, starting with the least knowledgeable player, we select some 
accumulation
point of the triplet $\bigl( s_4^n(\emptyset), U_4(L, s^{n}_{-4} \mid \emptyset ), U_4(R, s^{n}_{-4} \mid \emptyset ) \bigr)_{n \in \bN}$ to $\big( s^*_4(\emptyset),\rho_4[L],\rho_4[R]\big)$.

Now, move on to the second-least informed player, Player 3, who has partial knowledge. We need to take care that the construction of $s^*_3$ would leave open the door for construction of 
$s^*_{1,2}$
with $\rho_4[L] = U_4( L, s^*_{-4}   \mid \emptyset)$ and $\rho_4[R] = U_4( R, s^*_{-4}   \mid \emptyset)$. 
To this end, for each $t_3 \in \Ts_3$, we choose an 
accumulation point that is
consistent with the convergence we have already established of the $9$-tuple of $s_3^n(t_3)$ and $U_j(a_3,a_4,s^n_{1,2} \mid t_3)$ for each $j=3,4$ and each 
$a_3,a_4 \in \{L,R\}$,
where $U_4(\cdot \mid t_3)$ means $U_4(\cdot \mid \kappa_3(t_3))$. 
That is, we keep track not only of Player 3's strategy, but also of the expected payoffs to Player 3 \emph{and} Player 4 for each action \emph{profile} of these players. 
Denote a chosen 
accumulation 
point of this $9$-tuple as $s_3^*(t_3)$ and $\rho_3[j,a_3,a_4](t_3)$ for $j=3,4$ and
$a_3,a_4 \in \{L,R\}$.
Like in the previous example, it is not sufficient to choose any 
accumulation 
points across different $t_3 \in \Ts_3$; we must take care that these 
accumulation 
points are chosen so that if Player $4$ imagines the setup in which it is only her and Player $3$, and payoffs are given by $\rho_3$, then her expected payoff for an action 
$a_4 \in \{L,R\}$
when Player 3 uses $s_3^*$ is precisely $\rho_4[a_4]$, i.e., for 
$a_4 \in \{L,R\}$,
$\rho_4[a_4] = \int_{\Ts_3} \rho_3[4,s_3^*(t_3),a_4] {\mathbb P}(dt_3) $.
This establishes a certain consistency between $\rho_4$ and $\rho_3$ with $s_3^*$. 

In the next step, 
consider
Player 1 and 2, 
who have 
full
knowledge
because $t_3 = \kappa_2(t_1)$ and $\kappa_3(t_3) = \emptyset$.
We need for each $t_1$ to choose an 
accumulation 
 point of $(s_{1,2}^{n}(t_1))_{n \in \bN}$ 
 that is
 consistent with previous selections for this particular $\kappa_2(t_1)$, i.e., along subsequences of indices which give the chosen 
accumulation 
points $s_4^*(\emptyset)$, $s_3^*(\kappa_2(t_1))$, and $\rho_3[j,a_3,a_4]$ of $s_4^n(\emptyset)$, $s_3^n(\kappa_2(t_1) )$, and $U_j( a_3,a_4,s^{n}_{1,2}(t_1) \mid \kappa_2(t_1) )$ for $j=3,4$
and $a_3,a_4 \in \{L,R\}$,
respectively.
Furthermore, we need that the selection be done across all $t_1 \in \Ts_1 = \Ts_2$, so that, for each $t_3 \in \Ts$, the expected payoffs to Player $j=3,4$ under 
$s^*_{1,2}$ 
for any pair of actions 
$a_3,a_4\in\{L,R\}$
they may play, $U_j(a_3,a_4,s^*_{1,2} \mid t_3)$, agrees with $\rho_3[j,a_3,a_4](t_3)$. 
To do it for all $t_3 \in \Ts_3$ in parallel in a measurable fashion, we
appeal to 
the
Measurable ``Measurable Choice" Theorem of \cite{Me87}.

The proof in general is a formalization of the ideas above, although we note that we go player-by-player, without `bunching together' players who have identical information. 

\color{black}

\section{Proof of Theorem~\ref{thm:0}}
\label{section:proof}

The proof of Theorem~\ref{thm:0} consists of two steps.
In Section \ref{subsec: preliminary benchmarks}, we prove the existence of a Bayesian $\ep$-equilibrium under Condition \textbf{A1}, for every $\ep>0$.
In Section \ref{subsec: proof 1} we use this result to prove the existence of a Bayesian 0-equilibrium under both \textbf{A1} and \textbf{A2}. 
\color{black}

\subsection{Existence of a Bayesian $\ep$-equilibrium under \textbf{A1}}\label{subsec: preliminary benchmarks}

In this section we prove the existence of a Bayesian $\ep$-equilibrium under \textbf{A1}, for every $\ep>0$. 
To this end, we approximate
the information structure in a way that
\color{black}
is related to the approximation used by Shmaya and Solan \cite{Shmaya2004}. 
In Section  \ref{subsec: delta-approximation} we review the notion of $\delta$-approximation defined over a compact set in a metric space. 
This notion is applied in Section \ref{subsec: information} to approximate the information structure of the players in $\Gamma$ (under the nested information assumption and  \textbf{A1}). 
We then define a new game, which is identical to $\Gamma$ except that the information of the players is replaced by its approximation. 
In Section \ref{subsec: preliminary step}, we show that this approximated game has a Bayesian $0$-equilibrium,
and that this Bayesian 0-equilibrium is a Bayesian $\epsilon$-equilibrium of the original game $\Gamma$. 

\color{black}

\subsubsection{\label{subsec: delta-approximation} $\delta$-approximations}

We begin by recalling definitions related to dense subsets.

\begin{definition}[$\delta$-dense subset]\label{def: delta dense}
Let $U$ be a set in a metric space $\left(\mathcal{M},d\right)$, and fix $\delta > 0$.
A set $V\subseteq U$ is $\delta$-dense in $U$
if 
for every $u \in U$ there is $v \in V$ such that $d(u,v) < \delta$.
\end{definition}

When $U$ is contained in a compact set,
the $\delta$-dense set $V$ that we will consider will be implicitly assumed to be finite.
In this case,
there exists a measurable mapping $v_\delta:U\rightarrow V$ such that $d\left(u,v_\delta(u)\right) < \delta$ for every $u\in U$.
For every $u \in U$, the image $v_\delta(u)$ is called a \emph{$\delta$-approximation} of $u$ (by $V$). 
When $\Omega$ is a measurable space and $y : \Omega \to U$ is measurable, the mapping $v_\delta(y(\cdot)) : \Omega \to V$ will be a \emph{measurable} $\delta$-approximation of $y$.

For every random vector 
$Z : \cT \to \bR^d$ with finite range $\cZ$,
and every $i \in \cN$,
denote the conditional distribution of $Z$ given $t_i$ by 
\begin{equation*}
{\mathbb P}\left(Z\mid t_i\right)\equiv({\mathbb P}(Z=z\mid t_i))_{z\in\mathcal{Z}}.
\end{equation*} 
Thus, ${\mathbb P}\left(Z\mid t_i\right)$ is a  vector determined by $t_i$ and contained in the $\left(|\mathcal{Z}|-1\right)$-dimensional simplex. 
Since the $\left(|\mathcal{Z}|-1\right)$-dimensional simplex is compact,
there exists a $\delta$-approximation $\varphi(\cdot)(t_i)$ which belongs to the $\left(|\mathcal{Z}|-1\right)$-simplex and satisfies
\[
    \sum_{z\in \cZ} \bigl|{\mathbb P}\left(Z=z \mid t_i\right)-\varphi(z)(t_i)\bigr|<\delta\,,\  {\mathbb P}\text{-a.s.}
\]
\color{black}
\subsubsection{A finite approximation of the information structure}
\label{subsec: information}

In this section we present a finite approximation of the information structure,
which is suited to games where information is nested.

Fix a Bayesian game with nested information $\Gamma$ that satisfies 
\textbf{A1}.
That is, the action spaces $(\cA_i)_{i \in \cN}$ are finite and the payoff function is bounded by $M$. 
Since in this section we are interested in proving the existence of a Bayesian $\ep$-equilibrium,
we can assume w.l.o.g.~that the range 
of the payoff function $R$ is finite.
Namely, there is a finite collection $\cR$ of functions from $\cA$ to $[-M,M]^{n}$ such that 
for \emph{every} $t \in \cT$, the function $R(t) = (R_i(t))_{i \in \cN}$ is an element of $\cR$.

For every 
$i \in \cN$, 
every strategy $s_i \in \cS_i$, and every action $a_i \in \cA_i$, denote by $s_i(a_i)(t_i)$ the probability that Player~$i$ selects the action $a_i$ under $s_i$ when her type is $t_i$.
For every strategy profile $s \in \mathcal{S}$ and every action profile $a \in \cA$, 
the probability under $s$ that $a$ is selected by the players is
the random variable $p_s(a)$ defined by
\begin{equation*}
    p_s(a)(t)\equiv\prod_{i=1}^n s_i(a_i)(t_i), \ \ \forall t \in \cT.
\end{equation*}
Define
\[ p_{s_{-i}}(a_{-i})(t_{-i})\equiv \prod_{j \neq i} s_j(a_j)(t_j), \ \ \forall i \in \cN, s \in \cS, a \in \cA,  t \in \mathcal{T}.\]

We are now going to recursively define approximations of the information that the players have.
Fix $\delta>0$.
Let $\psi_1 : \mathcal{T} \to\mathcal{R}$ be the random vector defined by
\[\psi_1(t) \equiv R(t), \ \ \forall t \in \mathcal{T}. \]
Denote $r\equiv\left|\mathcal{R}\right|$, so that ${\mathbb P}(\psi_1 \mid t_1)$ is in the $(r-1)$-dimensional standard simplex, ${\mathbb P}$-a.s. 
Equip $\mathbb{R}^r$ with the $L_1$-norm. 
Let $\varphi_1\equiv\varphi_1(\cdot)(t_1)$ be a measurable 
$\delta$-approximation of ${\mathbb P}(\psi_1 \mid t_1)$, 
so that $\varphi_1$ has a finite range and
\[
    \sum_{z\in \cR} \bigl|{\mathbb P}\left(\psi_1=z \mid t_1\right)-\varphi_1(z)(t_1)\bigr|<\delta,\ \ {\mathbb P}\text{-a.s.}
\]

For $i \in \mathcal{N} \setminus \{1\}$,
suppose we have already defined random vectors $\psi_1,\varphi_1,\dots,\psi_{i-1},\varphi_{i-1}$, all with finite ranges, where $\psi_j:T_j \rightarrow \prod_{k=1}^{j-1} \cD_k  \times \cR$ and $\varphi_j:T_j \rightarrow \Delta(\prod_{k=1}^{j-1} \cD_k  \times \cR) \subseteq \mathbb{R}^{\prod_{j=1}^{i-1} |\cD_j| \times r}$, where for each $1 \leq j \leq i-1$,  $\mathcal{D}_j \subseteq \Delta(\prod_{k=1}^{j-1} \cD_k  \times \cR)$ is the range of $\varphi_j$. 
Define 
\begin{equation}\label{eq: random object}
\psi_i(t)\equiv\bigl(\varphi_1(t_1),\varphi_2(t_2),\ldots,\varphi_{i-1}(t_{i-1}), R(t)\bigr)  \in \prod_{j=1}^{i-1} \cD_j  \times \cR , \ \ {\mathbb P}\text{-a.s.},
\end{equation}
which is a random vector with a discrete distribution. 
The range of $\psi_i$ is finite and contains at most $\prod_{j=1}^{i-1} |\cD_j| \times r$ elements.
Equip $\mathbb{R}^{\prod_{j=1}^{i-1} |\cD_j| \times r}$ with the $L_1$-norm, so the $\left(\prod_{j=1}^{i-1} |\cD_j| \times r-1\color{black}\right)$-dimensional simplex is compact.
Let $\varphi_i$ be a measurable $\delta$-approximation
 of ${\mathbb P}\left( \psi_i \mid t_i\right)$,
and hence
\begin{equation}\label{equ:approx:1}
    \sum_{z\in\prod_{j=1}^{i-1}\cD_j\times\mathcal{R}}\bigl|{\mathbb P}\left(\psi_i=z\mid t_i\right)-\varphi_i(z)(t_i)\bigr|<\delta,\ \ {\mathbb P}\text{-a.s.}
\end{equation}
Thus, the range of 
$\varphi_i$ is contained in 
a 
finite $\delta$-dense subset of
the $\left(\prod_{j=1}^{i-1} |\cD_j| \times r-1\color{black}\right)$-dimensional simplex.

The random vectors $\varphi_1,\varphi_2,\ldots,\varphi_n$ have an intuitive interpretation. 
\begin{enumerate}
    \item[-] $\varphi_1(t_1)$ is an approximation of the information that Player~1 has on the payoffs at $t$. 

    \item[-] $\varphi_2(t_2)$ is an approximation of the information that Player~2 has at $t$ on the payoffs and on the approximated information that Player~1 has on the payoffs.

    \item[-] $\varphi_3(t_3)$ is an approximation of the information that Player~3 has at $t$ \color{black} on (i) the payoffs, (ii) the approximated information that Player~1 has on the payoffs, and (iii) the approximated information that Player~2 has on the payoffs and on the approximated information that Player~1 has on the payoffs.
    \item[-] Etc. 
\end{enumerate}

For every $i\in\mathcal{N}$ and every Borel function $f:\prod_{j=1}^{i-1}\cD_j\times\mathcal{R}\rightarrow\mathbb{R}$, 
let ${\mathbb E}_{\varphi_i}\left[f\right]$ be the random variable given by
\begin{equation*}
{\mathbb E}_{\varphi_i}\left[f\right](t_i)\equiv\sum_{z\in\prod_{j=1}^{i-1}\cD_j\times\mathcal{R}}f(z)\cdot\varphi_i(z)(t_i), \ \ {\mathbb P}\text{-a.s.}
\end{equation*}

The next lemma relates ${\mathbb E}_{\varphi_i}\left[f\right]$
to the conditional expectation of $f$ given $t_i$.

\begin{lemma}\label{lemma: approximation}
    Fix $i\in\mathcal{N}$ and let $f:\prod_{j=1}^{i-1}\cD_j\times\mathcal{R}\rightarrow\mathbb{R}$ be a real-valued Borel function which is bounded by $M>0$. Then,
    \begin{equation*}
        \bigl|{\mathbb E}\left[f(\psi_i)\mid t_i\right]-{\mathbb E}_{\varphi_i}\left[f\right](t_i)\bigr|<M\delta, \ \ {\mathbb P}\text{-a.s.}
    \end{equation*}
\end{lemma}

\begin{proof}
The claim holds since ${\mathbb P}$-a.s.~we have
\begin{align*}
\bigl|{\mathbb E}\left[f(\psi_i)\mid t_i\right]-{\mathbb E}_{\varphi_i}\left[f\right](t_i)\bigr|&
\leq\sum_{z\in\prod_{j=1}^{i-1}\cD_j\times\mathcal{R}}\left|f(z)\right|\cdot\bigl|{\mathbb P}\left(\psi_i=z\mid t_i\right)-\varphi_i(z)(t_i)\bigr|\\
&\leq M\sum_{z\in\prod_{j=1}^{i-1}\cD_j\times\mathcal{R}}\bigl|{\mathbb P}\left(\psi_i=z\mid t_i\right)-\varphi_i(z)(t_i)\bigr|
< M\delta,\nonumber
\end{align*}
where the second inequality holds by \eqref{equ:approx:1} and the assumption that $f$ is bounded by $M$.
\end{proof}
\newline\newline
For each $i\in\mathcal{N}$, let 
\begin{equation*}
\tau_i\equiv\tau_i(t_{i:n}) :\equiv\varphi_{i:n}\left(t_{i:n}\right)=(\varphi_i(t_i), \varphi_{i+1}(t_{i+1}),\dots,\varphi_n(t_n))\,.    
\end{equation*}
Intuitively, $\tau_i$ represents the approximated information of Player~$i$: 
Player~$i$ knows $\varphi_i$,
and since information is nested,
she also knows $\varphi_{i+1},\dots,\varphi_n$. 
The following result, which follows by the construction of $\left(\tau_i\right)_{i\in\mathcal{N}}$, 
details some properties of this approximated information structure.

\begin{lemma}
    $\text{ }$\begin{enumerate}
\item[\textbf{P1}.]  For each $i\in\mathcal{N}$,  $\tau_i$ has finite image.

\item[\textbf{P2}.]  For every $1\leq i\leq j\leq n$, $\tau_j$ is determined by $\tau_i$.

\item[\textbf{ P3}.]  For each $i\in\mathcal{N}$, $\tau_i(t_{i:n})$ is determined by $t_i$.

\item[\textbf{P4}.]  For each $i\in\mathcal{N}$, $\varphi_i(t_i)$ is determined by $\tau_i(t_{1:n})$, ${\mathbb P}$-a.s.
\end{enumerate}
\end{lemma}

\begin{proof}
\textbf{P1}, \textbf{P2}, and \textbf{P4} follow from the construction. 
We show that \textbf{P3} holds as well.
Fix then $i\in\mathcal{N}$, and 
recall that 
since the information is nested, $t_{i:n}$ is determined by $t_i$. Thus, $\tau_i(t_{1:n})$ is determined by $t_i$ and hence \textbf{P3} follows. 
\end{proof}

\begin{remark}\label{remark: tau_i(t_i)}
    \normalfont Due to \textbf{P3}, from now on, for each $i\in\mathcal{N}$, we shall write $\tau_i\equiv\tau_i(t_i)\equiv\tau_i(t_{i:n})$.
\end{remark}

\subsubsection{Proof: existence of Bayesian $\ep$-equilibrium under \textbf{A1}}\label{subsec: preliminary step}

We now define a Bayesian game $\widetilde \Gamma$,
which is similar to $\Gamma$, 
except that the information available to each player $i \in \cN$ is $\tau_i$ rather than $t_i$. 
Specifically, $\widetilde{\Gamma}$ is given by
\begin{itemize}
    \item A finite set of players $\widetilde{\mathcal{N}}\equiv\mathcal{N}=\{1,2,\ldots,n\}$.

    \item For every $i\in\mathcal{N}$, the set of types of Player $i$ is the range of $\tau_i=\varphi_{i:n}$, i.e., 
    $\cT_i\equiv \prod_{j=i}^n\mathcal{D}_j$ and hence $\widetilde {\mathcal{T}}\equiv\prod_{i\in\mathcal{N}}\widetilde{T}_i$.

    \item A common prior distribution $\widetilde{\mathbb P}$ on $\widetilde{\mathcal{T}}$,
    which is the push-forward probability measure induced by $\tau_{1:n}$ with respect to the probability measure ${\mathbb P}$:
    \[\widetilde {\mathbb P}(\widetilde t) \equiv {\mathbb P}(\tau_{1:n} = \widetilde t), \ \ \ \forall \widetilde t \in \widetilde\cT. \]
Denote by $\widetilde{\mathbb E}[\cdot]$ the expectation operator that corresponds to $\widetilde{\mathbb P}$.
    \item For each $i\in\mathcal{N}$, the set of actions available to Player $i$ is $\widetilde{\mathcal{A}}_i\equiv\mathcal{A}_i$,
    so that $\widetilde \cA = \cA$.

    \item 
For each $i \in \cN$, a measurable payoff function $\widetilde{R}_i: \widetilde{\mathcal{T}} \times \cA \to \bR$ given by
\begin{equation*}
    \widetilde{R}_i\left(\widetilde{t},a\right)\equiv \mathbb{E}\left[R_i(a)|\tau_{1:n}=\widetilde{t}\ \right],
\end{equation*}
for every $a\in{\mathcal{A}}$ and $\widetilde{t}\in\widetilde{\mathcal{T}}$ for which $\widetilde{\mathbb{P}}(\widetilde{\mathcal{T}}=\widetilde{t})>0$.
When $\widetilde{\mathbb{P}}(\widetilde{\mathcal{T}}=\widetilde{t})=0$,
the definition of $\widetilde R_i(\widetilde t,a)$ is irrelevant.
\color{black}
For each $i \in \cN$,
denote by $\widetilde{R}_i(a) : \mathcal{T} \to \bR$ the $a$-section of $\widetilde{R}_i$, for each $a \in \cA$;
and by $\widetilde{R}_i(\widetilde{t}) : \cA \to \bR$ the $\widetilde{t}$-section of $\widetilde{R}_i$, for each $\widetilde{t} \in \cT$.
We also set $\widetilde{R} \equiv (\widetilde{R}_i)_{i \in \cN}$.
\end{itemize}
The tower rule implies that for every $a\in\mathcal{A}$ and $i\in\mathcal{N}$,
\begin{equation*}
    \widetilde{\mathbb{E}}[\widetilde R_i(\widetilde t,a)]=\mathbb{E}[R_i(t,a)].
\end{equation*}
For each $i\in\mathcal{N},$ $\tau_{i:n}$ is determined by $\tau_i$, 
and hence for every $i\in\mathcal{N}$, $a\in\mathcal{A}$, and $\widetilde t = (\widetilde t_i)_{i \in I} 
\in\widetilde{\mathcal{T}}$ for which 
$\widetilde{\mathbb{P}}(\widetilde t)>0$,
the tower rule also implies that 
\begin{equation}\label{eq: approximated game conditional payoffs}
 \widetilde{\mathbb{E}}\left[\widetilde R_i(a)|\widetilde{t}_i\right]=\mathbb{E}\left[R_i(a)|\tau_i=\widetilde{t}_i\right].
\end{equation}
This means that the expected payoff of Player $i$ in $\widetilde{\Gamma}$ given her type 
$\widetilde{t}_i$ 
is the same as her expected payoff in $\Gamma$ given the event $\{\tau_i=\widetilde{t}\}$. 

A strategy of Player~$i$ in $\widetilde\Gamma$ is a function $\widetilde s_i : \widetilde T_i \to \Delta(\cA_i)$.
Such a function can be interpreted as a strategy in $\Gamma$
that is determined by $\tau_i$
(which, in turn, is determined by $t_i$).
Together with \eqref{eq: approximated game conditional payoffs},
this implies that 
a strategy profile
$\widetilde s=(\widetilde s_i)_{i\in\mathcal{N}}$ is a Bayesian $0$-equilibrium in $\widetilde{\Gamma}$ if and only if for every $i\in\mathcal{N}$
and every strategy $s_i \in \mathcal{S}_i$ determined by $\tau_i$, 
    \begin{equation}
    \label{equ:bayesian Gamma tilde:eq}
        U_i(\widetilde s)(t_i)=\mathbb{E}\left[R_i(\widetilde s)|\tau_i\right]\geq \mathbb{E}\left[R_i(s_i,\widetilde s_{-i})|\tau_i\right]=U_i(s_i,\widetilde s_{-i})(t_i), \ \ {\mathbb P}\text{-a.s.}
    \end{equation}

Since the sets of types in $\widetilde\Gamma$ are finite,
this game
\color{black}
admits a Bayesian 0-equilibrium.

The next lemma concerns the original game $\Gamma$,
and states that if each Player $j \in\mathcal{N} \setminus \{i\}$ adopts a strategy $s_j$ which is determined by $\tau_j$ (i.e., a strategy which is also feasible to her in $\widetilde{\Gamma}$),
then Player~$i$ has a $(\delta M|\cA|)$-best response which is determined by $\tau_i$ (i.e., a $\delta M|\cA|)$-best response in $\Gamma$ which is also feasible to her in $\widetilde{\Gamma}$). 
In view of \eqref{equ:bayesian Gamma tilde:eq}, 
this implies that every Bayesian $0$-equilibrium in $\widetilde{\Gamma}$ is a Bayesian $(\delta M|\cA|)$-equilibrium in $\Gamma$. 


\begin{lemma}
\label{lemma:br}
Suppose that \textbf{A1} holds.
Let $i \in \cN$ be a player, and let $s_{-i}\equiv\left(s_j\right)_{j\in\mathcal{N}_{-i}}\in\mathcal{S}_{-i}$ be a strategy profile such that $s_j$ is determined by $\tau_j$ for every $j\in\mathcal{N} \setminus \{i\}$. 
Then $\sup_{s_i\in\mathcal{S}_i} U_i(s_i,s_{-i} \mid t_i)$ is a random variable,
and there exists $s^*_{i}\in\mathcal{S}_i$ which is determined by $\tau_i$ such that
\begin{equation*}
U_i(s^*_i,s_{-i} \mid t_i)\geq \sup_{s_i\in\mathcal{S}_i} U_i(s_i,s_{-i} \mid t_i)-\delta M\left|\mathcal{A}\right|,  \ \ {\mathbb P}\text{-a.s.}
\end{equation*}
\end{lemma}

\begin{proof}
Player~$i$'s (random) expected payoff given her information, 
when she selects action $a_i\in\mathcal{A}_i$ and the other players follow the strategy profile $s_{-i}$, 
is the random variable $m_i(a_i)$ determined by $t_i$ and defined as 
\begin{equation}
m_{i}(a_i)(t_i)\equiv\sum_{a_{-i}\in\mathcal{A}_{-i}}{\mathbb E}\left[R_i(a_i,a_{-i})p_{s_{-i}}(a_{-i})\mid t_i\right], \ \ {\mathbb P}\text{-a.s.}
\label{equ:8.1}
\end{equation}  
By \textbf{A1}, $\mathcal{A}_i$ is a finite set, 
hence the best payoff that Player~$i$ can achieve is the random variable $\widetilde m_i$ determined by $t_i$ and defined by
\begin{equation*}
\widetilde m_i(t_i)\equiv\max\left\{m_i(a_i)(t_i)\ \colon \ a_i\in\mathcal{A}_i\right\}, \ \ \forall t_i \in \cT_i.    
\end{equation*}
In particular, $\sup_{s_i\in\mathcal{S}_i} U_i(s_i,s_{-i} \mid t_i) = \widetilde m_i$ is a random variable.

Let us show that for each $a_i\in\mathcal{A}_i$, there exists a random variable $\widehat{m}_i(a_i)$ which is determined by $\tau_i$ and such that
  \begin{equation}\label{eq:approximation eta}
     \left|\widehat{m}_i(a_i)-m_i(a_i)\right|< \delta M\left|\mathcal{A}\right|, \ \ {\mathbb P}\text{-a.s.}
  \end{equation}
Before proving the existence of such random variables $(\widehat m_i(a_i))_{a_i\in \cA_i}$,
we will show how the lemma follows from their existence.
First, denote 
\begin{equation}
\widehat{m}_i\equiv\max\left\{\widehat{m}_i(a_i)\ \colon \ a_i\in\mathcal{A}_i\right\},
\end{equation}
which is a random variable determined by $\tau_i$,
and notice that by \eqref{eq:approximation eta}, $\left|m_i-\widehat{m}_i\right|<\delta M\left|\mathcal{A}\right|$, ${\mathbb P}$-a.s. 
Each of the random variables $(\widehat m_i(a_i))_{a_i\in \cA_i}$ is determined by $\tau_i$, and hence the set of maximizers $\arg\max\left\{\widehat{m}_i(a_i);a_i\in\mathcal{A}_i\right\}$ is finite and also determined by $\tau_i$. 
Therefore, 
there is a Borel
selector\footnote{ Let $X,Y$ be topological spaces
and $\Phi:X \Rightarrow Y$ a correspondence (a set-valued mapping). 
A \emph{Borel selector} of $\Phi$ is a Borel mapping  $f:X \rightarrow Y$ such that $f(x) \in \Phi(x)$ for all $x \in X$.
} $A_i$ of $\arg\max\left\{\widehat{m}_i(a_i);a_i\in\mathcal{A}_i\right\}$, 
that is, 
\[ A_i\in \arg\max\left\{\widehat{m}_i(a_i)\ \colon \ a_i\in\mathcal{A}_i\right\}, \ \ {\mathbb P}\text{-a.s.},\]
which is determined by $\tau_i$; 
hence $A_i$ satisfies the requirements of the lemma.  



We turn to prove the existence of $\widehat m_i(a_i)$. Recall that 
$\psi_i = (\varphi_1,\varphi_2,\ldots,\varphi_{i-1},R)$.
The strategy $s_j$ is determined by $\tau_j$ for every $j \in \cN \setminus \{i\}$,
and by \textbf{P2}  it is also determined by $\tau_1$.
Therefore, 
$p_{s_{-i}}\left(a_{-i}\right)$ is also determined by $\tau_1=\varphi_{1:n}$,
for every $a_{-i}\in\mathcal{A}_{-i}$. 
By~\eqref{equ:8.1}, there exists a measurable function $f_{a_{-i}} : \text{Supp}\left(\psi_i,\varphi_{i:n}\right) \to [-M,M]$ such that
\begin{equation}
\label{equ:eta}
    m_i(a_i)(t_i)= \sum_{a_{-i}\in\mathcal{A}_{-i}}{\mathbb E}\left[f_{a_{-i}}\left(\psi_i,\varphi_{i:n}\color{black}\right)\mid t_i\right], \ \ {\mathbb P}\text{-a.s.}
\end{equation}
By definition, $\tau_i=\varphi_{i:n}$. Hence, by \textbf{P3}, 
 $\varphi_{i:n}$
\color{black}
is also determined by $t_i$. 
Therefore,
\begin{equation}\label{eq: conditional}
     m_i(a_i)(t_i)=\sum_{a_{-i}\in\mathcal{A}_{-i}}\int {\mathbb E}\left[f_{a_{-i}}\left(\psi_i,x\right)\mid t_i\right]\chi_{\varphi_{i:n}(t_i)}\,({\rm d}x), \ \ {\mathbb P}\text{-a.s.},
\end{equation}
where $\chi_{\varphi_{i:n}(t_i)}$ is the Dirac measure concentrated at $\left\{\varphi_{i:n}(t_i)\color{black}\right\}$. 
Define

\begin{align}\label{eq: eta hat}
\widehat{m}_i(a_i)(t_i)&\equiv\sum_{a_{-i}\in\mathcal{A}_{-i}}\int {\mathbb E}_{\varphi_i}\left[f_{a_{-i}}\left(\,\cdot\,,x\right)\right](t_i)\,\chi_{\varphi_{i:n}(t_i)\color{black}}\, {\rm d}x\\
&=
\sum_{a_{-i}\in\mathcal{A}_{-i}} {\mathbb E}_{\varphi_i}\left\{f_{a_{-i}}\left[\,\cdot\,,\varphi_{i:n}(t_i)\color{black}_{i:n}(t_i)\right]\right\}(t_i), \ \ {\mathbb P}\text{-a.s.},\nonumber
\end{align}
and notice that $\widehat{m}_i(a_i)(t_i)$ is determined by $\varphi_i(t_i)$ and $\varphi_{i:n}(t_i)\color{black}$, for every realization of types profile $t \in \cT$.
Since both $\varphi_i$ and $\varphi_{i:n}(t_i)\color{black}$
are determined by $\tau_i$, 
$\widehat{m}_i(a_i)$ is also determined by $\tau_i$. 
Finally, 
by \eqref{eq: conditional}, \eqref{eq: eta hat}, and Lemma \ref{lemma: approximation},
\begin{align*}
&\left|\widehat{m}_i(a_i)(t_i)-m_i(a_i)(t_i)\right|
\\&< 
\sum_{a_{-i}\in\mathcal{A}_{-i}}
\int 
\Bigl|\color{black}{\mathbb E}\left[f_{a_{-i}}\left(\psi_i,x\right)\mid t_i\right]
- {\mathbb E}_{\varphi_i}\left[f_{a_{-i}}\left(\,\cdot\,,x\right)\right](t_i)\Bigr|\color{black}\,\chi_{\varphi_{i:n}(t_i)\color{black}}({\rm d}x)
\\
&<\delta M\left|\mathcal{A}_{-i}\right|, \ \ {\mathbb P}\text{-a.s.},
\end{align*}
 and~\eqref{eq:approximation eta} follows. 
\end{proof}

\bigskip

\subsection{Existence of a Bayesian $0$-equilibrium under \textbf{A1} and \textbf{A2}}\label{subsec: proof 1} 

In this section we derive Theorem~\ref{thm:0} from the existence of Bayesian $\ep$-equilibria for $\ep > 0$,
which exist by the results in Section~\ref{subsec: preliminary benchmarks}.
We will fix a sequence $(s^k)_{k \in \bN}$ of Bayesian $\frac{1}{k}$-equilibria, and show that this sequence has a measurable accumulation point which is a Bayesian $0$-equilibrium. 
In Section \ref{subsec: selectors} we present tools related to the existence of measurable selections and their integration. 
In Section \ref{subsec:expected_conditional_payoffs}, we represent the conditional expected payoff in a useful way. 
In Section \ref{subsec:proof_Psi_def}, 
we define the  correspondences $(\Psi_i)_{i \in \cN}$ of the accumulation points of $(s^k)_{k \in \bN}$ and study some of their properties, 
In Section \ref{subsec:proof_Psi_eq}, we characterize Bayesian $0$-equilibria in terms of $(\Psi_i)_{i \in \cN}$. 
In Sections \ref{subsec:proof_Psi_selections} and \ref{subsec:proof_build_eq} we show that there exist measurable selections of $(\Psi_i)_{i \in \cN}$ satisfying the characterization of the $0$-equilibrium presented in Section \ref{subsec:proof_Psi_eq}.  
\color{black}

\subsubsection{Selectors and Integration}\label{subsec: selectors}

Let $X,Y$ be standard Borel spaces,%
\footnote{A \emph{standard Borel space} is a topological space homeomorphic to a Borel subset of a Polish space.} 
and $\Phi:X \Rightarrow Y$ a correspondence (a set-valued mapping). 
We say that $\Phi$ has nonempty compact values if $\Phi(x)$ is nonempty and compact for every $x \in X$.
The following classical result provides topological conditions that guarantee the existence of a Borel selector.%
\footnote{\cite{kuratowski1965general} states the measurability assumption on $\Phi$ 
in a way that, for nonempty compact valued correspondences, is equivalent to the one we provided here; see, e.g., \cite{Himmelberg1975}.}

\begin{theorem}[Kuratowski and Ryll-Nardzewski \color{black} \cite{kuratowski1965general}]\label{thm:selector}
	Suppose the correspondence $\Phi:X \Rightarrow Y$ has a Borel graph and nonempty compact values. 
	Then, $\Phi$ has a Borel selector.
\end{theorem}

Let $\textbf{S}_\Phi$ denote the collection of all Borel selectors of $\Phi$. 
Suppose $Y$ is a subset of a Euclidean space, and let ${\mathbb P}$ be a finite Borel measure on $X$.  The \emph{Aumann integral} of $\Phi$ (with respect to ${\mathbb P}$) is 
\[
\int_X \Phi(x) {\mathbb P}(dx) \equiv \left\{ \int_X f(x) {\mathbb P}(dx) \mid f \in \textbf{S}_\Phi \right\}.
\]

The following result appears in \cite{Aumann1965Integrals} and the references therein when $X=[0,1]$ and ${\mathbb P}$ is the Lebesgue measure; 
the general case follows by minor modifications, or from more general results, like Theorem \ref{thm:mertens} below. 

\begin{theorem}\label{thm:aumann_integral_closure}
Suppose the correspondence $\Phi:X \Rightarrow \bR^n$ is bounded,\footnote{That is, there is a bounded $W \subseteq \bR^n$ such that $\Phi(x) \subseteq W$ for every $x \in X$.} and has a Borel graph and nonempty compact values. Then $\int_X \Phi(x) {\mathbb P}(dx)$ is nonempty and compact.
\end{theorem}

We need the following slight generalization of Theorem \ref{thm:aumann_integral_closure}.

\begin{proposition}\label{prop:aumann_integral_closure}
Let $X$ be a standard 
\color{black} Borel space, and let $Y$ be a compact metric space.
Suppose the correspondence $\Phi \colon X \Rightarrow Y$ has a Borel graph and nonempty compact values. 
Let $\zeta \colon Y \rightarrow \bR^n$ be continuous. 
Then the set
	\begin{equation}\label{eq: correspondence}
 \left\{ \int_X \zeta \circ f(x) {\mathbb P}(dx) \colon f \in \textbf{S}_\Phi  \right\} \subset \bR^n
\end{equation}
is nonempty and compact.
\end{proposition}



\begin{proof}
	Define $\Psi:X \rightarrow \bR^n$ by $\Psi = \zeta \circ \Phi$, i.e., $\Psi(\cdot) = \zeta(\Phi(\cdot))$. 
Since $Y$ is  compact, $\zeta$ is continuous, and $\Phi$ has nonempty compact values, 
it follows that $\Psi$ is bounded and has nonempty compact values. 
We contend that 
	\[
	\left\{ \zeta \circ f \colon f \in \textbf{S}_\Phi  \right\} = 	 \textbf{S}_\Psi,
	\]
	from which the proposition will follow due to Theorem \ref{thm:aumann_integral_closure}. 
	Clearly, $\left\{ \zeta \circ f \colon f \in \textbf{S}_\Phi  \right\} \subseteq 	 \textbf{S}_\Psi$, as $\zeta \circ f \in  \textbf{S}_\Psi$ for each $f \in \textbf{S}_\Phi$. 
Conversely, suppose $g \in  \textbf{S}_\Psi$. 
Since $Y$, the domain of $\zeta$, is a nonempty compact set and $\zeta$ is continuous, 
the correspondence $\zeta^{-1}(\cdot) : \Psi(X) \to Y$ has a Borel graph and nonempty compact values. Thus,  Theorem \ref{thm:selector}, applied to the correspondence $\zeta^{-1}(\cdot)$, 
yields a Borel mapping $\zeta'\colon\mathrm{Image}(\zeta) \rightarrow Y$ such that $\zeta \circ \zeta' = \mathrm{id}$. 
	Hence, $f:= \zeta' \circ g$ satisfies $g = \zeta \circ f$ and $f \in \textbf{S}_\Phi$.
\end{proof}

\bigskip

When $(f_n)_{n=1}^\infty$ is a sequence of mappings
between two topological spaces $X$ and $Y$,
we denote 
\color{black}
by $\overline{\text{Lim}}((f_n)_n) \colon X \Rightarrow Y$  
the correspondence such that $\overline{\text{Lim}}((f_n)_n)(x)$ is the set of all accumulation points of $(f_n(x))_{n=1}^\infty$,
for every $x \in X$. 
When $X$ and $Y$ are standard Borel spaces with $Y$ compact, this correspondence has a Borel graph with nonempty compact values, see \cite[Prop. 10.1]{Me87}.

\begin{lemma}\label{lemma:aumann_integral_closure}
Let $(X, P)$ be a standard Borel measure space, 
let $Y$ be a compact metrizable space, and for each $n \in \bN$, let $f_n \colon X \rightarrow Y$ be measurable. 
Let $\zeta \colon Y \rightarrow \bR^n$ be continuous, and suppose
	\[
	\int_X \zeta(f_n(x)) {\mathbb P}(dx) \xrightarrow[n\to\infty]{} z^*.
	\]
	Then, there is a Borel selector $f^*\colon X \rightarrow Y$ of the correspondence $\overline{\mathrm{Lim}}((f_n)_n)$ such that
	\[
	z^* = \int_X \zeta(f^*(x)) {\mathbb P}(dx).
	\]
\end{lemma}

When $Y \subseteq \bR^n$ and $\zeta = \mathrm{id}$, 
Lemma \ref{lemma:aumann_integral_closure} 
was proven in, e.g., 
\cite[p. 69]{Himmelberg1975}, or \cite{Aumann1976}.%
\footnote{\cite{Aumann1976} addresses the case where ${\mathbb P}$ is non-atomic; the case where ${\mathbb P}$ may have atoms follows by passing to a sequence $(f_n)_{n=1}^\infty$ which converges on atoms, by a diagonalization construction.} 

\bigskip

\begin{proof}
Denote for simplicity 
$\mathcal{L} = \overline{\text{Lim}}((f_n)_n)$ 
and
$\widehat{\mathcal{L}}=\overline{\text{Lim}}((\zeta \circ f_n)_n)$. Applying the result in the restricted case $Y \subseteq \bR^n$ and $\zeta = \mathrm{id}$ to the series $(\zeta \circ f_n)_n$ shows that there is Borel selector $g^*$ of $\widehat{\mathcal{L}}$ such that
		\[
	z^* = \int_X g^*(x) {\mathbb P}(dx).
	\]
	We claim that%
\footnote{ In fact, there is equality in~\eqref{eq: L}. However, we only need this inclusion.} 
 \begin{equation}\label{eq: L}
    \widehat{\mathcal{L}}(x) \subseteq \zeta(\mathcal{L}(x)), \ \ \forall x\in X.
    \end{equation} 
Indeed, if $y \in \widehat{\mathcal{L}}(x)$, 
	then there are indices $(n_k)_{k \in \bN}$ such that $\lim_{k \to \infty} \zeta(f_{n_k}(x)) = y$. 
	Letting $z$ be an accumulation point of $(f_{n_k}(x))_{k \in \bN}$, which exists by compactness and metrizability of $Y$, 
 we deduce by the continuity of $\zeta$ that $z \in \mathcal{L}(x)$ and $y = \zeta(z)$. 
    
As in the proof of Proposition~\ref{prop:aumann_integral_closure},    
the correspondence $\zeta^{-1}$ has a Borel graph and nonempty compact values, 
and hence so does $\zeta^{-1}(g^*(\cdot))$. 
Thus, Eq.~\eqref{eq: L} implies that for every $x\in X$, $\zeta^{-1}(g^*(x))$ is a nonempty compact subset of $\mathcal{L}(x)$. 
It remains to apply Theorem~\ref{thm:selector} to the correspondence $\zeta^{-1}(g^*(\cdot))$.  
\end{proof}

\bigskip
	
The following result is a slight generalization of 
the Measurable ``Measurable Choice'' Theorem 
from \cite{Me87}, adapted to a bounded Borel setting;%
	\footnote{A similar but weaker result is proven in \cite{Artstein1989},
which only gives an ``almost everywhere" type of selection.} 
	\cite{Me87} deals with the case 
$W \subseteq \bR^n$,
	$\zeta = \mathrm{id}$, and $B = Y \times Z$.
	
\begin{theorem}\label{thm:mertens}
Let $Y$ and $Z$ be Borel spaces, and let $F$ be a  correspondence from a Borel set $B \subseteq Y \times Z$ to a compact metric space $W$, with nonempty compact values and a Borel graph. 
Let $\zeta\colon W \rightarrow \bR^n$ be continuous, and let $q$ be a Borel transition kernel from $Y$ to $Z$,%
\footnote{A \emph{transition kernel} is a map from $Y$ to $\Delta(Z)$, the space of Borel probability distributions on $Z$, such that for each Borel $B \subseteq Y$, the mapping $z \rightarrow q(B \mid z)$ is Borel.} 
such that $q(B_y \mid y) = 1$  for each $y \in Y$, 
where $B_y \equiv \{z \in Z \colon (y,z) \in B \}$ is the $y$-section of $B$. 
Define a correspondence $F^\diamond \colon Y \Rightarrow\bR^n$ by
\begin{equation}\label{eq:def_F_diamond}
F^\diamond(y) \equiv \left\{ \int_{B_y} \zeta \circ f(y,z) q(dz \mid y) \colon f \in \textbf{S}_{F(y,\cdot)} \right\}.
\end{equation}
Then:
\begin{itemize}
\item $F^\diamond$ is bounded, and has nonempty compact values and a Borel graph $\mathrm{Gr}(F^\diamond)$.
\item There is a Borel mapping $g \colon \mathrm{Gr}(F^\diamond) \times Z \rightarrow \bR^n$ such that for each $(y,u) \in \mathrm{Gr}(F^\diamond)$ and each $z \in B_y$, we have $g(y,u,z) \in \zeta\circ F(y,z)$ and
			\[
			u = \int_Z \zeta \circ g(y,u,s) q(ds \mid y).
			\]
		\end{itemize}
	\end{theorem}

\begin{proof}
As mentioned, \cite{Me87} proved the case $B = Y \times Z$ (in which case $B_y=Z$ for each $y \in Y$), $W \subseteq \bR^n$, and $\zeta = \mathrm{id}$. 
To prove the general case, 
fix $x_0 \in \bR^n$, and define $\widetilde{F}:Y \times Z \Rightarrow \bR^n$ by $\widetilde{F} = \zeta \circ F$ on $B$, and $\widetilde F \equiv \{ x_0 \}$ outside of $B$.
The argument from the proof of Proposition~\ref{prop:aumann_integral_closure}, together with the fact that $q(B_y \mid y) = 1$ for each $y \in Y$, shows that 
\[
F^\diamond(y) = \left\{ \int_Z \widetilde{f}(y,z) q(dz \mid y) \colon \widetilde{f} \in \textbf{S}_{\widetilde{F}(y,\cdot)
} \right\}.
\]
Since $W$ is compact, 
$F^\diamond$ is bounded,
and
by Proposition \ref{prop:aumann_integral_closure}, 
it has nonempty compact values.
By \cite{Me87}, $F^\diamond$ is measurable,
and hence the first bullet holds.

By \cite{Me87} once again,
$F^\diamond$ has a Borel selector, i.e., 
a Borel mapping $\widetilde{g} \colon \mathrm{Gr}(F^\diamond) \times Z \rightarrow \bR^n$ such that $g(y,u,z) \in \widetilde{F}(y,z)$
for each $(y,u) \in \mathrm{Gr}(F^\diamond)$ and $z \in Z$ (and hence $g(y,u,z) \in \zeta \circ F(y,z)$ for $z \in B_y$), and
	\[
	u = \int_Z \widetilde{g}(y,u,s) q(ds \mid y).
	\]
By applying the same arguments as in  the proof of Lemma~\ref{lemma:aumann_integral_closure}, 
there is Borel mapping $\zeta'\colon\mathrm{Image}(\zeta) \rightarrow W$ such that $\zeta \circ \zeta' = \mathrm{id}$. 
Setting $g = \zeta' \circ \widetilde{g}$ yields the mapping indicated in the second bullet.
\end{proof}

\subsubsection{Conditional Expected Payoffs}
\label{subsec:expected_conditional_payoffs}
 
Recall that $M$ is a bound on the payoffs in the game. 
For each $i \in \cN$, set $\mathcal{P}_{i:n}\equiv[-M,M]^{(n-i+1) \times |\cA_{i:n}|}$,
so that $R(t) \in \cP_{1:n}$ for every $t \in \cT$.
\color{black}
Each vector $\rho_{i:n} \in \mathcal{P}_{i:n}$ corresponds to a vector of payoff functions
for the set of players $[i\!:\!n]$,
where the players' actions are $(\cA_j)_{j=i}^n$.
For such a vector, we denote by $\rho_j(a_{i:n})$
the coordinate that corresponds to Player~$j$
and to the action profile $a_{i:n}$.
The multilinear extension of $\rho_j$ is still denoted by $\rho_j$,
so that
\[ \rho_j(x_{i:n}) \equiv \sum_{a_{i:n} \in \cA_{i:n}}\rho_j(a_{i:n}) \prod_{k=i}^n x_k(a_k), \ \ \ 
\forall x_{i:n} \in \cX_{i:n}. \]

Denote by $ U_j(s_{1:i}, a_{i+1:n} | t_k)$ the expected payoff of Player~$j$ when players $[1\!:\!i]$ follow the strategies $s_{1:i}$ and players $[i+1\!:\!n]$ select the actions $a_{i+1:n}$, 
given that Player~$k$'s type is $t_k$.
It will be convenient to denote by 
\[ U_{i+1:n}(s_{1:i},\cdot | t_{i+1}) \equiv 
\bigl( U_{i+1:n}(s_{1:i}, a_{i+1:n} | t_{i+1})\bigr)_{a_{i+1:n} \in \cA_{i+1:n}} \in \cP_{i+1:n}\]
the payoff function of players $[i+1\!:\!n]$ that is induced by the strategy profile $s_{1:i}$ and the type $t_{i+1}$.
The multilinear extension of $U_{i+1:n}(s_{1:i},\cdot | t_{i+1})$ is still denoted by $U_{i+1:n}(s_{1:i},\cdot | t_{i+1})$,
and it is a mapping
from $\cX_{i+1:n}$ to $\bR^{n-i}$.


\color{black}

\subsubsection{The Limit of $\frac{1}{k}$-Equilibria: The Operator $\Psi_i$}\label{subsec:proof_Psi_def}

Since information is nested, $t_i$ determines $t_{i+1}, t_{i+2}, \ldots,t_n$.
For convenience, when $j > i$ and $s_j$ is a strategy of Player~$j$, we will sometimes write
$s_j(t_i)$ instead of $s_j(\kappa_{j-1} \circ \kappa_{j-2} \circ\dots\circ\kappa_{i}(t_i))$. 

Fix 
a sequence $(s^k)_{k=1}^\infty$ of strategy profiles such that  $s^k\equiv (s_1^k,\ldots,s_n^k)$ is a $\frac{1}{k}$-Bayesian equilibrium, for each $k \in \bN$,
which in turn is guaranteed by the result obtained in Section \ref{subsec: preliminary benchmarks}.
%
We would like to prove that an appropriate limit of $(s^k)_{k=1}^\infty$
is a Bayesian 0-equilibrium.
To this end, 
we will consider for each $i \in \cN$ the accumulation points of the sequence 
$\Big( s^k_{i:n}(t_i)), 
U_{i:n}(s^k,\cdot| t_i), U_{i+1:n}(s^k,\cdot| \kappa_i(t_{i}))
\Big)_{k=1}^\infty$,
and show that proper Borel selectors of these correspondences (indexed by $i$)
induce a Bayesian 0-equilibrium.

Define a correspondence $\Psi_n:\cT_n \Rightarrow \cX_{n} \times \cP_n$ by
\begin{align}\label{eq:Psi_def_n}
\Psi_n(t_n) \equiv 
\left\{ (x_n,\rho_n) \in \cX_n \times \cP_n
\mid  (x_n,  \rho_n) \in 
\overline{\rm Lim}\Big(\Big( s^k_{n}(t_n), 
U_n(s^k| t_n) \Big)_k\Big)   \right\},
\end{align}
and, for every $i < n$,
\color{black}
define the correspondence $\Psi_i:\cT_i \times \cX_{i+1:n} \times \cP_{i+1:n} \Rightarrow \cX_{i:n} \times \cP_{i:n}$ by
\begin{align}\label{eq:Psi_def}
&\Psi_i(t_i, x_{i+1:n}, \rho_{i+1:n} ) \equiv\\ 
&\Big\{ ((x_i,x_{i+1:n}),\widehat\rho_{i:n}) \in \cX_{i:n} \times \cP_{i:n}
\nonumber\\
&\ \ \ \ \ \ \ 
\mid  ((x_i,x_{i+1:n}),  \widehat\rho_{i:n},\rho_{i+1:n}) \in 
\overline{\rm Lim}\Big(\Big( s^k_{i:n}(t_i), 
U_{i:n}(s^k_{1:i-1},\cdot| t_i), U_{i+1:n}(s^k_{1:i},\cdot| \kappa_i(t_{i}))
\Big)_k\Big)   \Big\}.
 \nonumber
\end{align}
Note that $\Psi_i$ may have empty values.
This happens, for example, when $x_{i+1:n}$ is not an accumulation point of $(s^k_{i+1:n}(t_i))_{k=1}^\infty$.
The definition implies that if $(\widehat x_{i:n},\widehat\rho_{i:n}) \in \Psi_i(t_i,x_{i+1:n},\rho_{i+1:n})$,
then $\widehat x_{i+1:n} = x_{i+1:n}$.
The relation between $\widehat\rho_{i:n}$ and $\rho_{i+1:n}$
is more complex and we will not need it.
\color{black}


The following lemma holds from the definitions and the fact that each $\cP_i$ is compact.

\begin{lemma}
\label{lemma:6}
For $i\in \cN$, $\Psi_i$ has a Borel graph and  nonempty compact values.
\end{lemma}
If $(x_{1:n},\widehat\rho_{1:n}) \in \Psi_1(t_1,x_{2:n},\rho_{2:n})$,
then $\widehat\rho_{1:n}$ is a vector of payoff functions for all players.
In particular, the $j$'s coordinate $\widehat\rho_j$ of $\widehat\rho_{1:n}$ satisfies%
\footnote{Below, $U_j(a|t_1) = U_j(s|t_1)$,
where $s = (s_i)_{i \in \cN}$, is the strategy profile where each player selects the action $a_i$ at all types.}
\begin{equation}
\label{equ:rho1}
\widehat\rho_j(a) = U_j(a|t_1) = R_j(t,a), 
\end{equation}
for all $a \in \cA_{1:n}$,
where $t = (t_1, \kappa_1(t_1), \dots, \kappa_{n-1} \circ \dots \circ\kappa_2\circ\kappa_1(t_1))$.

\subsubsection{Equilibrium Characterization via $\Psi_1,\Psi_2,\ldots,\Psi_n$}\label{subsec:proof_Psi_eq}

In this section we will provide a characterization of Bayesian 0-equilibria in terms of the mappings 
$\Psi_1,\Psi_2,\ldots,\Psi_n$.
\color{black}

\begin{lemma}\label{prop:rho_argmax}
For each $i \in \cN$, ${\mathbb P}$-a.e.~$t_i \in \cT_i$, $x_{i+1:n} \in \cX_{i+1:n}$, and $\rho_{i+1:n} \in \cP_{i+1:n}$, 
if $(x_{i:n}, \widehat\rho_{i:n}) \in \Psi_i(t_i, x_{i+1:n}, \rho_{i+1:n})$, then
	\begin{equation}
 \label{equ:rho_argmax}
	x_i \in \argmax_{y_i \in \cX_i} \widehat\rho_i(y_i,x_{i+1:n}).
	\end{equation}
\end{lemma}
The expression under the argmax 
on the right-hand side of~\eqref{equ:rho_argmax} is the expected payoff of Player $i$  under the payoff function 
$\widehat\rho_{i}$ when Players $[i\!:\!n]$ use mixed actions $y_i,x_{i+1},\ldots,x_n$, respectively. 
Lemma~\ref{prop:rho_argmax},
which follows by continuity arguments,
is the only place in the proof where the fact that $s^k$ is a $\frac{1}{k}$-equilibrium, for each $k \in \bN$ is directly used.

\bigskip

\begin{proof}
Fix a player $i<n$, $t_i \in \cT_i$, $x_{i+1:n} \in \cX_{i+1:n}$, $\rho_{i+1:n} \in \cP_{i+1:n}$, and $\widehat\rho_{i:n} \in \cP_{i:n}$ such that $(x_i, \widehat\rho_{i:n}) \in \Psi_i(t_i, x_{i+1:n}, \rho_{i+1:n})$.
By assumption, there is a sequence of indices $(k_l)$ such that 
\[
\lim_{l \to \infty} \Big(s^{k_l}_{i:n}(t_i),   U_{i:n}(s^{k_{l}}_{1:i-1},\cdot| t_i) \Big) = ((x_i, x_{i+1:n}),  \widehat\rho_{i:n} ).
\]
For each $l=1,2,\ldots$, the expected payoff of Player $i$ with type $t_i$, playing mixed action $y_i \in \cX_i$ while the others follow strategy profile $s^{k_l}$ is, 
\begin{align}
U_i(s^{k_l}_{-i},y_i|t_i) = \sum_{a_i \in \cA_i} \sum_{a_{i+1:n} \in \cA_{i+1:n}}  y_i(a_i) \cdot \Big(\prod_{j=i+1}^n s^{k_l}_j(t_i)(a_j) \Big) \cdot U_i(s^{k_l}_{1:i-1},a_{i:n}|t_i). \nonumber
\end{align}
Since $s^{k_l}$ is a $\frac{1}{k_l}$-equilibrium, 
there is a set $\Xi_l \subseteq \mathcal{T}_i$ with ${\mathbb P}(\Xi_l)=0$ such that for every $t_i \notin \Xi_l$,
\[
U_i(s^{k_l}_{-i},y_i|t_i) \leq U_i(s^{k_l}|t_i) + \frac{1}{k_l},
\]
that is,
\begin{align}
\sum_{a_i \in \cA_i}& \sum_{a_{i+1:n} \in \cA_{i+1:n}}  y_i(a_i) \cdot \Big(\prod_{j=i+1}^n s^{k_l}_j(t_i)(a_j) \Big) \cdot U_i(s^{k_l}_{1:i-1},a_{i:n}|t_i) \label{equ:prop2}\\
& \leq \frac{1}{k_l} + \sum_{a_i \in \cA_i} \sum_{a_{i+1:n} \in \cA_{i+1:n}}  s_i^{k_l}(a_i) \cdot \Big(\prod_{j=i}^n s^{k_l}_j(t_i)(a_j) \Big) \cdot U_i(s^{k_l}_{1:i-1},a_{i:n}|t_i). \nonumber
\end{align}
Taking $l \rightarrow \infty$ gives that for every $t_i \in \mathcal{T}_i \backslash \bigcup_{l \in \mathbb{N}} \Xi_l$, 
\begin{align}
\sum_{a_i \in \cA_i}& \sum_{a_{i+1:n} \in \cA_{i+1:n}}  y_i(a_i)\cdot \Big(\prod_{j=i}^n x_j(a_j) \Big) \cdot \widehat\rho_i(a_{i:n}) \nonumber\\
& \leq \sum_{a_i \in \cA_i} \sum_{a_{i+1:n} \in \cA_{i+1:n}}  x_i(a_i) \cdot \Big(\prod_{j=i}^n x_j(a_j) \Big) \cdot \widehat\rho_i(a_{i:n}), \nonumber
\end{align}
and the claim follows.
The proof for $i=n$ is similar yet simpler,
since in~\eqref{equ:prop2} the inner summations on the two sides are vacuous.
\end{proof}

\begin{corollary}\label{cor:equilibrium_psi}
Let $s\in\mathcal{S}$ be a strategy profile.
Suppose that for ${\mathbb P}$-almost every $t_n \in \cT_n$
\[ (s_n(t_n), U_n(s_n|t_n)) \in \Psi_n(t_n), \]
and for each $i<n$ and ${\mathbb P}$-almost every $t_i \in \cT_i$ \color{black},
\begin{equation}\label{eq:equilibrium_condition_Psi}
\bigl(s_{i:n}(t_i), U_{i:n}(s|t_i)\bigr) \in 
\Psi_i\bigl( t_i, s_{i+1:n}(t_i), U_{i+1:n}(s|t_{i+1}) \bigr).
\end{equation}
Then, $s$ is a Bayesian $0$-equilibrium.
\end{corollary}


\subsubsection{Selections from $(\Psi_i)_{i \in \cN}$}
\label{subsec:proof_Psi_selections}

Recall that $\kappa_i:\mathcal{T}_i\rightarrow\mathcal{T}_{i+1}$
is the mapping that indicates the type of Player~$i+1$
for each type of Player~$i$.
In particular, 
$\kappa^{-1}_{i}(t_{i+1}) \equiv \{ t_{i}\in \cT_i \mid \kappa_i (t_{i}) = t_{i+1} \}$
is the set of all types of Player~$i$ that are consistent with the type of Player~$i+1$.

An element $(x_{i:n},\rho_{i:n}) \in \cX_{i:n} \times \cP_{i:n}$
is a pair consisting of a mixed action vector for players $[i\!:\!n]$
and a payoff function for a game restricted to these players.
It will prove convenient to denote, for each $j \in [i:n]$, Player~$j$'s payoff in this game by
\[ \gamma_j(x_{i:n},\rho_{i:n}) \equiv \rho_j(x_{i:n}), \]
and set 
\[ \gamma_{i:n}(x_{i:n},\rho_{i:n}) \equiv (\gamma_j(x_{i:n},\rho_{i:n}))_{j=i}^n. \]
\color{black}


The next lemma intuitively states that every point $(t_i,x_{i:n},\rho_{i+1:n}, \widehat\rho_{i:n})$ in the graph of $\Psi_i$ 
can be extended to a point in the graph of $\Psi_{i-1}$.
\color{black}

\begin{lemma}\label{prop:mertens_groundwork}
Fix $i=2,3,\ldots,n$,  $t_i \in \cT_i$, $x_{i+1:n} \in \cX_{i+1:n}$, and  $\rho_{i+1:n} \in \cP_{i+1:n}$. 
If $(x_{i:n}, \widehat\rho_{i:n}) \in \Psi_i( t_i, x_{i+1:n}, \rho_{i+1:n}  )$, 
then $\Psi_{i-1}( t_{i-1}, x_{i:n},  \widehat\rho_{i:n} ) \neq \emptyset$,
for each $t_{i-1} \in \kappa^{-1}_{i-1}(t_i)$;
For $i=n$, the terms $x_{i+1:n}$ and $\rho_{i+1:n}$ are vacuous.
Moreover, there exists a Borel mapping $f:\cT_{i-1} \rightarrow \cX_{i-1:n} \times \cP_{i-1:n}$ such that  
\begin{equation}\label{eq:mertens_groundwork_eq1}
f(t_{i-1}) \in  \Psi_{i-1}(t_{i-1}, x_{i:n},  \widehat\rho_{i:n} ), \ \ \ \forall t_{i-1} \in \kappa^{-1}_{i-1}(t_i),
\end{equation}
and
\begin{equation}\label{eq:mertens_groundwork_eq2}
\widehat\rho_{i:n} = \int_{\cT_{i-1}} 
\gamma_{i:n}(f(t_{i-1})) {\mathbb P}(dt_{i-1} \mid t_i).
\end{equation}
\end{lemma}

Eq.~(\ref{eq:mertens_groundwork_eq1}) states that for any fixed $(x_{i:n},\widehat\rho_{i:n})$, 
on $\kappa^{-1}_{i-1}(t_i)$, $f(\cdot)$ is a selector of the correspondence
\[
t_{i-1} \Rightarrow \Psi_{i-1}( t_{i-1}, x_{i:n},
\widehat\rho_{i:n}),
\]
which, by Lemma~\ref{lemma:6}, has a Borel graph and nonempty compact values.

\bigskip

\begin{proof}
Suppose that $(x_{i:n}, \widehat\rho_{i:n}) \in \Psi_i( t_i, x_{i+1:n}, \rho_{i+1:n} )$. 
Then there is a sequence of indices $(k_l)_{l=1}^\infty$ such that  
\begin{equation}
\label{equ:94}
x_{i:n} = \lim_{l \to \infty} s_{i:n}^{k_l}(t_i) \hbox{\ \ \  and \ \ \ } 
\widehat\rho_{i:n} = \lim_{l \to \infty} U_{i:n}(s^{k_l},\cdot|t_i).
\end{equation} 

Applying Lemma \ref{lemma:aumann_integral_closure}
to $X =\cT_{i-1}$,
${\mathbb P}(\mathrm{d} x) = {\mathbb P}(\mathrm{d} t_{i-1}|t_i)$,
$Y = \cX_{i:n} \times \cP_{i:n}$,
$f_l(t_i) = (s^{k_l}_{i:n}(t_i),U_{i:n}(s^{k_l}_{i:n}|t_i)$ for each $l \in \bN$,
$\zeta=\gamma_{i:n}$,
and $z^* = (x_{i:n},\widehat\rho_{i:n})$,
we conclude that 
there is a Borel selector $f$ of the correspondence 
\[ t_{i-1} \Rightarrow \overline{\rm Lim}\left( \left(s^{k_l}_{i-1:n}(t_{i-1}),  U_{i-1:n}(s^{k_l}_{1:i-1},\cdot|t_{i-1})\right)_l \right) \]
such that  (\ref{eq:mertens_groundwork_eq2})  holds.  
Condition~\eqref{equ:94} implies that 
\begin{align}
\emptyset \neq	\overline{\rm Lim} \left( \left(s^{k_l}_{i-1:n}(t_{i-1}), U_{i-1:n}(s^{k_l}_{1:i-1},\cdot|t_{i-1})\right)_l \right) \subseteq \Psi_{i-1}(t_{i-1}, x_{i:n},  \widehat\rho_{i:n}), \nonumber
\end{align}
which completes the proof.
\end{proof}

\bigskip

The next result is a measurable version of Lemma~\ref{prop:mertens_groundwork}.

\begin{lemma}\label{cor:mertens}
For each $i=2,\ldots,n$,  there is a mapping $f_{i-1}:\cT_{i-1} \times \mathrm{Gr}(\Psi_i) \rightarrow \cX_{i-1:n} \times \cP_{i-1}$ such that
for each $(t_i, x_{i+1:n}, \rho_{i+1:n},  x_{i:n}, \widehat\rho_{i:n}) \in \mathrm{Gr}(\Psi_i)$,
\begin{equation}\label{eq:mertens_inclusion}
	f_{i-1}(t_{i-1}, t_i, x_{i:n}, \rho_{i+1:n}, \widehat\rho_{i:n}) \in  \Psi_{i-1}( t_{i-1}, x_{i:n}, \widehat\rho_{i:n})\ ,\  \forall t_{i-1} \in \kappa^{-1}_{i-1}(t_i),
	\end{equation}
	and
	\begin{equation}\label{eq:mertens_equality}
    \widehat\rho_{i:n} = \int_{\cT_{i-1}}  \gamma_{i:n}\bigl(f_{i-1}(t_{i-1}, t_i, x_{i:n}, \rho_{i+1:n}, \widehat \rho_{i:n})\bigr) {\mathbb P}(dt_{i-1} \mid t_i).
	\end{equation}
For $i=n$, 
in both \eqref{eq:mertens_inclusion}
and~\eqref{eq:mertens_equality}
the term $\rho_{i+1:n}$ is vacuous. 
%
\end{lemma}

\begin{remark}\normalfont[Dependence of $f$ on $\rho_{i+1:n}$]
$\text{ }$ The dependence of $f$ on $\rho_{i+1:n}$ seems superfluous, but may be indispensable. 
This dependence parallels the dependence of the equilibria in stochastic games 
on the previous state \emph{and} on the current state,
whose existence was proved by \cite{MePa91} using the Measurable ``Measurable Choice" Theorem of \cite{Me87}.
\end{remark}

\begin{remark}\label{rem:f_x_j}\normalfont
By the properties of $\Psi_{i-1}$, if 
\[ (\widetilde{x}_{i-1:n}, \widetilde{\rho}_{i-1:n}) = 	f_{i-1}(t_{i-1}, t_i, x_{i:n}, \rho_{i+1:n}, \widehat\rho_{i:n}), \]
then $\widetilde{x}_{i:n} = x_{i:n}$.
Later we will make use of this observation.
\end{remark}

\noindent\textbf{Proof of Lemma~\ref{cor:mertens}:}
Fix $i=2,\ldots,n$, and 
apply Theorem~\ref{thm:mertens}
with the following parameters:
\begin{itemize}
	\item $Y = \mathrm{Gr}(\Psi_i)$.
	\item $Z = \cT_{i-1}$.
	\item $B =  \{ (t_{i-1}, t_i, x_{i+1:n}, \rho_{i+1:n}, x_{i:n},\widehat\rho_{i:n}) \in Z \times Y  \mid t_i = \kappa_{i-1}(t_{i-1}) \}$.
	\item $q(t_{i-1} \mid t_i, x_{i+1:n}, \rho_{i+1:n}, x_{i:n},\widehat\rho_{i:n}) = {\mathbb P}(t_{i-1} \mid t_i)$.
	\item $W = \cX_{i-1:n} \times \cP_{i-1:n}$.
\item $\zeta$\, is the evaluation map 
 $\gamma_{i:n}$ \color{black} on $\cX_{i-1:n} \times \cP_{i-1:n}$.
\item $F$ is defined on $B$ by $F(t_{i-1}, t_i, x_{i+1:n}, \rho_{i+1:n}, x_{i:n}, \widehat\rho_{i:n})\equiv \Psi_{i-1}(t_{i-1},x_{i:n}, \widehat\rho_{i:n})$.
	\item $F^\diamond:Y \rightarrow \cP_{i:n}$ as defined in (\ref{eq:def_F_diamond}).
\end{itemize}
By Theorem \ref{thm:mertens}, 
there is a Borel mapping $g:\mathrm{Gr}(F^\diamond) \times \cT_{i-1} \rightarrow \cX_{i-1:n} \times \cP_{i-1:n}$ 
such that for every $y = (t_i, x_{i+1:n}, \rho_{i+1:n}, x_{i:n}, \widehat\rho_{i:n}) \in Y = \mathrm{Gr}(\Psi_i)$,
every $u \in F^\diamond(y)$, 
and every $t_{i-1} \in \kappa^{-1}_{i-1}(t_i)$,
we have 
\[ g(y,u,t_{i-1}) \in F(t_{i-1},y) = \Psi_{i-1}(t_{i-1},x_{i:n}, \widehat\rho_{i:n}), \]
and 
\[
u = \int_{\cT_{i-1}} \gamma_{i:n}\bigl(g(y,u,t_{i-1})\bigr) {\mathbb P}(dt_{i-1} \mid t_i).
\]

It follows from Lemma~\ref{prop:mertens_groundwork} that for any $y = (t_i, x_{i+1:n}, \rho_{i+1:n}, x_{i:n}, \widehat\rho_{i:n}) \in Y$, 
noting that $B_y = \kappa^{-1}_{i-1}(t_i)$ and hence ${\mathbb P}(B_y \mid t_i)=1$, 
there is $f:\cT_{i-1} \rightarrow  \cP_{i-1:n}$ such that  $f|_{B_y}$ is a Borel selector of $t_{i-1} \rightarrow	F(t_{i-1},t_i,x_{i+1:n}, \rho_{i+1:n}, x_{i:n}, \widehat\rho_{i:n}) = \Psi_{i-1}(t_{i-1},x_{i:n}, \widehat\rho_{i:n})$ and such that  (\ref{eq:mertens_groundwork_eq2}) holds, which means that 
\[ \widehat\rho_{i:n} \in F^\diamond(t_{i-1}, t_i, x_{i+1:n}, \rho_{i+1:n}, x_{i:n}, \widehat\rho_{i:n}). \]
Defining $f_{i-1}:\cT_{i-1} \times \mathrm{Gr}(\Psi_i)  \rightarrow \cX_{i-1:n} \times \cP_{i-1:n}$ by
\[
f_{i-1}\big(t_{i-1},  t_i, x_{i:n}, \rho_{i+1:n}, \widehat\rho_{i:n} \big) \equiv g\big((t_i, x_{i+1:n}, \rho_{i+1:n}, x_{i:n}, \widehat\rho_{i:n}), \widehat\rho_{i:n}, t_{i-1} \big),
\]
yields the desired result. 
\endpf

\subsubsection{Construction of a Bayesian 0-Equilibrium}
\label{subsec:proof_build_eq}

In this section we define a strategy profile 
$s_* = (s_*^1,\ldots,s_*^n)$,
and prove that it is a Bayesian 0-equilibrium.
\color{black}

For any set $A$, denote by $\pi_A$ the projection map to $A$.
Let $f_{n-1},\ldots,f_1$ be the mappings given by Lemma~\ref{cor:mertens}. 
Define mappings $(g_i)_{i=1}^n$ recursively (backwards) as follows.
\begin{itemize}
\item Let $g_n:\cT_n \rightarrow \cX_n \times \cP_n$ be a Borel selector of $\Psi_n$, which exists by Theorem \ref{thm:selector} (recall that $\Psi_n$ depends only on the type of Player $n$).
Define 
\begin{equation}
\label{equ:88}
s^*_n \equiv \pi_{\cX_n} \circ g_n.
\end{equation}

\item For $i = 2,\ldots,n$,  assuming that  we have already defined mappings $(g_j)_{j=i}^n$,
let $g_{i-1}:\cT_{i-1} \rightarrow  \cX_{i-1:n} \times \cP_{i-1:n}$ be defined by
\begin{equation}\label{eq:mertens_applied}
\hspace{-1cm}
g_{i-1}(t_{i-1}) \equiv f_{i-1}\Big( t_{i-1}, 
\kappa_{i-1}(t_{i-1}), 
s^*_{i:n}(\kappa_{i-1}(t_{i-1})), 
\pi_{\cP_{i+1:n}} \circ g_{i+1}(\kappa_i(\kappa_{i-1}(t_{i-1}))), 
\pi_{\cP_{i:n}} \circ g_i(\kappa_{i-1}(t_{i-1})) \Big),
\end{equation}
where the penultimate argument is vacuous when $i=n-1$, and set
\begin{equation}
\label{equ:89}
 s^*_{i-1}\equiv\pi_{\cX_{i-1}} \circ g_{i-1}.
\end{equation}
\end{itemize}

For $i=1,\dots,n-1$, the mapping $g_i$ is well defined
provided 
\[ \Big(
\kappa_{i-1}(t_{i-1}), 
s^*_{i+1:n}(\kappa_{i-1}(t_{i-1})), 
\pi_{\cP_{i+1:n}} \circ g_{i+1}(\kappa_i(\kappa_{i-1}(t_{i-1}))), 
s^*_{i:n}(\kappa_{i-1}(t_{i-1})), 
\pi_{\cP_{i:n}} \circ g_i(\kappa_{i-1}(t_{i-1})) \Big) \]
always lies in $\mathrm{Gr}(\Psi_i)$. 
The next lemma states that this is indeed the case.

\begin{lemma}
The mappings $(g_i)_{i=1}^{n-1}$ are well defined.
\end{lemma}

\begin{proof}
Denote $t_i = \kappa_{i-1}(t_{i-1})$.
We will prove that 
for each $t_{i-1} \in \cT_{i-1}$, 
\begin{equation}
\label{equ:91}
    \Big( t_i, s^*_{i+1:n}(t_i),  \pi_{\cP_{i+1:n}} \circ g_{i+1}(\kappa_i(t_i)), s^*_{i:n}(t_i), \pi_{\cP_{i:n}} \circ g_i(t_i)\Big) \in \mathrm{Gr}(\Psi_i),
\end{equation} 
and
\begin{equation}
\label{equ:92}
s^*_{i:n}(t_i) = \pi_{\cX_{i:n}} \circ g_i(t_i).
\end{equation}
Note the slight difference between \eqref{equ:89} and \eqref{equ:92}: 
in the former we set 
$s^*_{j}(t_j)$ to be $\pi_{\cX_{j}} \circ g_{j}(t_j)$,
while the latter claims that $s^*_j(t_j) = \pi_{\cX_j} \circ g_i(t_i)$; 
by construction, using Remark \ref{rem:f_x_j}, these agree.

We prove \eqref{equ:91} and \eqref{equ:92} by induction, starting with $i=n$.
In this case, \eqref{equ:92} holds by \eqref{equ:88}.
Moreover, the left-hand side in~\eqref{equ:91} becomes
\[ \bigl(t_n,s^*_n(t_n), \pi_{\cP_n} \circ g_n(t_n)\bigr) 
= \bigl(t_n,\pi_{\cX_n} \circ g_n(t_n), \pi_{\cP_n} \circ g_n(t_n)\bigr) 
= \bigl(t_n, g_n(t_n)\bigr), \]
which lies in $\mathrm{Gr}(\Psi_i)$ since $g_n$ is a selector of $\Psi_n$.

Let now $i \in [1:n-1]$, 
and assume by induction that~\eqref{equ:91} and~\eqref{equ:92} hold for $i+1$.
By~\eqref{eq:mertens_inclusion}, \eqref{eq:mertens_applied}, and the induction hypothesis,
\begin{align}\label{eq:g_induction_hyp}
g_{i}(t_{i}) 
&\in \Psi_{i}\bigl( t_{i}, s^*_{i+1:n}(t_{i+1}), \pi_{\cP_{i+1:n}} \circ g_{i+1}(t_{i+1})\bigr),
\end{align}
where $t_{i+1} = \kappa_i(t^i)$.  It follows from the properties of $\Psi_i$ --- or from Remark \ref{rem:f_x_j} and the definition of $g_i$ in \eqref{eq:mertens_applied} --- that $\pi_{X_{i+1:n}} (g_i(t_i)) = s^*_{i+1:n}(t_{i+1})$. By definition,
$s^*_i(t^i) = \pi_{X_i}(g_i(t^i))$. Putting these together shows that \eqref{equ:92} holds for $i$.
Once we proved that~\eqref{equ:92} holds for $i$,
we have
\[ (s^*_{i:n}, \pi_{\cP_{i:n}}\circ g_{i}(t_{i}))
= (\pi_{\cX_{i:n}}\circ g_{i}(t_{i}), \pi_{\cP_{i:n}}\circ g_{i}(t_{i}))
=
g_{i}(t_{i}). \]
Hence, the left-hand side in~\eqref{equ:91} becomes
\[ \Big( t_i, s^*_{i+1:n}(t_{i+1}), \pi_{\cP_{i+1:n}} \circ g_{i+1}(\kappa_i(t_i)), g_i(t_i)\Big). \]
By the induction hypothesis (\ref{eq:g_induction_hyp}), this element is in $\text{Gr}(\Psi_i)$, as required.
\end{proof}
\color{black}

\bigskip
The next result relates $U_{i:n}(s^*,\cdot|t_i)$ to $g_i(t_i)$.

\begin{lemma}\label{lem:relate_g_p}
For each $i \in \cN$ and each $t_i \in \cT_i$,
\begin{equation}
U_{i:n}(s^*,\cdot|t_i)= \pi_{\cP_{i:n}} \circ g_i(t_i).
\end{equation}
\end{lemma}

\begin{proof}
We prove the claim by induction on $i$.
We start with $i=1$.
By~\eqref{eq:mertens_applied} and \eqref{eq:mertens_inclusion}, for each $t_1 \in \cT_1$,
	\[
	g_1(t_1) \in \Psi_1\big( t_1, s^*_{2:n}(t_1),\pi_{\cP_{2:n}}\circ g_{2}(\kappa_1(t_1) )  \big),
	\]
	which, by~\eqref{equ:rho1}, implies the result for $i=1$.

Fix now $i > 1$, 
and suppose the claim holds for $i-1$.
For each $t_i \in \cT_i$,
\begin{align}
U_{i:n}(s^*,\cdot|t_i)
& 	= \int_{\cT_{i-1}} \gamma_{i:n} \bigl(s^*_{i-1:n}(t_{i-1}), U_{i-1:n}(s^*_{1:i-1},\cdot|t_{i-1}) \bigr)  {\mathbb P}({\rm d} t_{i-1} \mid t_i ) \label{equ:81}\\
	& = \int_{\cT_{i-1}} \gamma_{i:n} \bigl(s^*_{i-1}(t_{i-1}), \pi_{\cP_{i-1:n}} \circ g_{i-1}(t_{i-1}) \bigr)  {\mathbb P}({\rm d} t_{i-1} \mid t_i ) \label{equ:82}\\
	& = \int_{\cT_{i-1}} \gamma_{i:n} \bigl(g_{i-1}(t_{i-1}) \bigr)  {\mathbb P}({\rm d} t_{i-1} \mid t_i ) \label{equ:83}\\
	& = \int_{\cT_{i-1}} \gamma_{i:n} \Big( f_{i-1}\big( t_{i-1}, t_i, s^*_{i:n}(t_i), \pi_{\cP_{i:n}} \circ g_i(t_i), \pi_{\cP_{i+1:n}} \circ g_{i+1}(\kappa_i(t_i)) \big)  \Big)  {\mathbb P}({\rm d} t_{i-1} \mid t_i ) \label{equ:84}\\
	& = \pi_{\cP_{i:n}} \circ g_i(t_i), \label{equ:85}
\end{align}
where \eqref{equ:81} holds since information is nested, 
\eqref{equ:82} holds by the induction hypothesis, 
\eqref{equ:83} holds by~\eqref{equ:92}, 
\eqref{equ:84} holds by~(\ref{eq:mertens_applied}), 
and \eqref{equ:85} holds by (\ref{eq:mertens_equality}).
\end{proof}

\bigskip
We can now conclude the proof of Theorem~\ref{thm:0}.

\begin{lemma}
$s^*$ is a Bayesian 0-equilibrium.
\end{lemma}

\begin{proof}
Fix $i \in \cN$ and $t_i \in \cT_i$. Then
\begin{align*}
\bigl(s^*_i(t_i),U_{i:n}(s^*|t_i)\bigr) 
= g_i(t_i)  
&\in \Psi_i\big( t_i, s^*_{i+1:n}(t_i),  \pi_{\cP_{i+1}} \circ g_{i+1}(\kappa_i(t_i)) \big)
\nonumber\\ 
&=  \Psi_i\big( t_i, s^*_{i+1:n}(t_i), U_{i+1:n}(s^*|\kappa_i(t_i))  \big),
\end{align*}
 where the first equality holds by~\eqref{equ:92} and Lemma \ref{lem:relate_g_p}, 
 the inclusion holds by~\eqref{eq:mertens_applied} and~\eqref{eq:mertens_inclusion},
 and the second equality holds by Lemma~\ref{lem:relate_g_p}.
Corollary \ref{cor:equilibrium_psi} now implies that $s^*$ is a Bayesian 0-equilibrium.
\end{proof}

\section{Discussion}
\label{section:discussion}

Our study raises several extensions and open problems, which we present in this section.

\paragraph*{Inconsistent beliefs}

We assumed that all players share the same belief ${\mathbb P}$ on $\mathcal{T}$.
Our result holds also when beliefs are inconsistent, namely, each player $i \in \cN$ holds a different belief ${\mathbb P}_i$ on $\mathcal{T}$.
In this case, each player's payoff is defined relative to ${\mathbb P}_i$ (rather than relative to ${\mathbb P}$).
The proof remains largely the same, requiring only minor adaptations.



\paragraph*{Weakening Assumption \textbf{A1}.}
We proved Theorem~\ref{thm:0} under the assumption that the action sets are finite, the payoffs are bounded, and the type spaces are Polish. One could ask whether the result still holds when weakening this assumption, namely,
requiring that (i) the action sets $\mathcal{A}_i$ are compact metric,
and
(ii) the payoff functions are integrable, and continuous in $a$ for every fixed vector of types.
In a subsequent work we prove that in this case a Bayesian $\ep$-equilibrium exists.
We do not know whether a Bayesian 0-equilibrium exists in this case as well.

\paragraph*{Tree-like information}
The information in the game is nested
if the players are ordered,
and each player knows the types of all players that follow her in that order.
Suppose that the players are vertices of a tree,
and each player knows the types of all her descendants. 
Does a Bayesian 0-equilibrium exist?
Our proof extends to this case, 
conditioned that the payoff of each player depends only on the types and actions of her descendants and her ancestors.

\paragraph*{Multi-stage Bayesian games}
We showed that nested information is a sufficient condition for the existence of a Bayesian $0$-equilibrium in a general class of \emph{single-stage} Bayesian games. 
A natural question regards the existence of a   Bayesian $0$-equilibrium in \emph{multi-stage} Bayesian games with nested information.

Specifically, a multi-stage Bayesian game with $m \geq2$ stages is similar to a Bayesian game as in Definition~\ref{def:bayesian:game}, 
except that the players play for $m$ stages and Player~$i$'s type is stage dependent.
That is, Player~$i$'s type is a vector $t_i = (t_i^1,t_i^2,\ldots,t_i^m)$,
the collection of types of the players are drawn at the outset,
and at each stage $1\leq k\leq m$, each player learns her own stage type $t_i^k$.
Player~$i$'s payoff at each stage $k$
depends on 
the players' stage-types $(t_i^k)_{i \in \cN}$, and the players stage actions.

Repeating the arguments of the current paper implies existence of $0$-Bayesian equilibrium in the multi-stage game whenever the payoff functions $(R_{i,t})_{i \in \cN, 1 \leq t \leq m}$ satisfy \textbf{A1} and \textbf{A2}, and the nested information assumption is replaced by the following two conditions:\color{black}
\begin{enumerate}
\item [\textbf{A3-a}] Information is nested:
    For each $k\in\{1,2,\ldots,m\}$ and each $i \in \{1,2,\dots,n-1\}$, 
    $t_{i+1}^k$ is a determined by $t_i^k$.
    
\item [\textbf{A3-b}] Information is revealed with delay of one stage: For every $k\in\{1,2,\ldots,m-1\}$, 
    $t_1^k$ is determined by $t_n^{k+1}$;
i.e., 
the information of Player~$1$ is available to Player $n$ with delay of one stage.  
\end{enumerate}
\color{black}
Models of multi-stage Bayesian games that satisfy these two assumptions have been studied
in the literature on control under the name  \emph{information structure with one-step-delay}, 
see, e.g., Aicardi et al.~\cite{Aicardi1987}, Nayyar et al.~\cite{Nayyar2010}, and Varaiya and Walrand \cite{Varaiya1978}. 
We conjecture that Conditions \textbf{A1}, \textbf{A2}, and \textbf{A3-a}  
(without \textbf{A3-b}) are not sufficient to guarantee the existence of a Bayesian $0$-equilibrium in multi-stage Bayesian games.  

\paragraph*{Stopping games with asymmetric information.}

One class of multi-stage Bayesian games is the class of stopping games with asymmetric information. 
In stopping games, players choose in each round to stop or continue; 
the game ends when at least one player chooses to stop, and the payoff profile is 
given by an $\bR^N$-valued stochastic process that depends on 
the set of players who chose to stop.
There is also some designated payoff in case the game never 
terminates.
Incomplete 
and
asymmetric information can be introduced by adding uncertainty on 
the payoff process.

Such games have been studied both in the framework of discrete time and that of continuous time, see, e.g., 
Gr\"un \cite{Grun},
Lempa and Matom\"aki \cite{LM2013},
Gensbittel and Gr\"un \cite{Gensbittel2019},  
Esmaeeli, Imkeller, and Nzengang \cite{Esmaeeli 2019},
Gapeev and Rodosthenous \cite{Gapaev 2021},
P\'erez {et al.}~\cite{Perez 2021},
Jacobovic \cite{Jacobovic2022},
and
De Angelis {et al.}~\cite{De Angelis2022, De Angelis2022b}.
The open problem we raised for multi-stage Bayesian games translates to the following: 
Does every stopping game (in discrete or continuous time) with finite horizon and information structure that satisfies \textbf{A3-a} (or a continuous-time analog) admit a Bayesian $0$-equilibrium?

\paragraph*{Nested information and the value of information.}

Various aspects of the value of information in two-player zero-sum Bayesian games have
been studied, e.g., by
De Meyer et al.~\cite{De Meyer2010}, Lehrer and Rosenberg \cite{Lehrer2006}, and Ui \cite{Ui2015}.
As we now argue, nested information is related to the study of the value of information in multiplayer Bayesian games with symmetric information.
Indeed, consider, say, the best equilibrium payoff of Player~$i$ in a multiplayer Bayesian game with symmetric information,
and her best equilibrium payoff in the same game, when she does not obtain any information (while the other players obtain the symmetric information).
As the latter game has nested information,
our result ensures that the measure suggested above (and other natural variants) is well defined.
\color{black}

\paragraph*{The universal belief space.}

The universal belief space is the space that contains all infinite belief hierarchies,
see Mertens and Zamir \cite{MZ 1985}.
As we have seen, 
when information is nested,
the players' belief hierarchies are determined by the first $n$ levels in the belief hierarchy.
It will be interesting to know whether there is a canonical form to the universal belief space in this case.

In Aumann's model of incomplete information,
the information of a player is given by a partition of the state space,
and the information structure is nested
if under some ordering of the players,
Player~$i$'s partition refines Player~$j$'s partition whenever $i < j$.
Call an information structure \emph{finite}
if each belief hierarchy is determined by its first $k$ levels,
for some $k \in \bN$.
Nested information structures are finite.
An example of a finite information structure that is not nested 
is that of a ``piecewise'' nested information structure:
the state space is divided into several common knowledge components,
and in each component, the information is nested,
possibly with a different ordering between the players.
Are the piecewise nested information structures \emph{all} finite information structures?
\color{black}


\end{document}